\documentclass[12pt]{article}
\usepackage{amssymb}
\usepackage{mathrsfs}
\usepackage[hypertex]{hyperref}
\usepackage{amsfonts}
\usepackage{graphicx}
\usepackage{latexsym,amsmath,amssymb,amsfonts,amsthm}

\newcommand{\bcen}{\begin{center}}
\newcommand{\ecen}{\end{center}}
\newtheorem{theorem}{Theorem}[section]
\newtheorem{lemma}[theorem]{Lemma}

\newtheorem{corollary}[theorem]{Corollary}

\newtheorem{remark}[theorem]{Remark}
\newtheorem{example}[theorem]{Example}

\newtheorem{defn}[theorem]{Definition}

\begin{document}
\setcounter{page}{1}
\title{Weighted heat kernel comparison theorems and its applications in spectral geometry}
\author{Jing Mao\\
\small{In memory of my father Mr. Xu-Gui Mao}}

\date{}
\protect \footnotetext{\!\!\!\!\!\!\!\!\!\!\!\!{~MSC 2020:
58C40, 58J50, 35K08, 35P15, 35J15.} }
\maketitle ~~~\\[-15mm]

\begin{center}
{\footnotesize Faculty of Mathematics and Statistics,\\
 Key Laboratory of Applied
Mathematics of Hubei Province, \\
Hubei University, Wuhan 430062, China\\
Key Laboratory of Intelligent Sensing System and Security (Hubei
University), Ministry of Education\\
Email: jiner120@163.com
 }
\end{center}


\begin{abstract}
In this paper, we firstly establish weighted heat kernel comparison
theorems for the weighted heat equation on complete manifolds with
radial curvatures bounded, and then by mainly using this conclusion,
we can obtain two eigenvalue comparison theorems for the first
Dirichlet eigenvalue of the Witten-Laplacian as applications in
spectral geometry. Besides, as byproducts, eigenvalue comparisons
for the first Dirichlet eigenvalue of the weighted $p$-Laplacian
($1<p<\infty$) have been established as well.
 \end{abstract}


\section{Introduction} \label{Sect1}
\renewcommand{\thesection}{\arabic{section}}
\renewcommand{\theequation}{\thesection.\arabic{equation}}
\setcounter{equation}{0}

Given an $n$-dimensional Riemannian manifold $M^{n}$ with the metric
$\langle\cdot,\cdot\rangle$, assume that $\phi\in C^{\infty}(M^n)$
is a smooth\footnote{~One might find that
  it is sufficient to get some
conclusions in this paper, provided $\phi$ is only $C^2$. However,
for the convenience and in order to avoid any potential confusion,
we assume that $\phi$ is smooth and wish to ignore the discussion on
seeking the weakest regularity of the potential function $\phi$ such
that all the assertions in this paper are still valid, which of
course is not our main purpose here. } real-valued function defined
on $M^n$, and $\Omega\subset M^n$ is a bounded domain. Then it is
natural to define the notion of smooth metric measure space
$(M^{n},\langle\cdot,\cdot\rangle,e^{-\phi}dv)$, which actually
endows the Riemannian volume element $dv$ of $M^{n}$ with a
conformal factor $e^{-\phi}$. On the smooth metric measure space
$(M^{n},\langle\cdot,\cdot\rangle,e^{-\phi}dv)$, the
Witten-Laplacian (also called weighted Laplacian, drifting
Laplacian, or $\phi$-Laplacian) and the $N$-dimensional
Bakry-\'{E}mery Ricci curvature tensor can be well-defined as
follows
\begin{eqnarray*}
\Delta_{\phi}:=\Delta-\langle\nabla\cdot,\nabla\phi\rangle=e^{\phi}\mathrm{div}(e^{-\phi}\nabla\cdot)
\end{eqnarray*}
and
\begin{eqnarray*}
\mathrm{Ric}^{N}_{\phi}:=\mathrm{Ric}+\mathrm{Hess}\phi-\frac{d\phi\otimes
d\phi}{N-n},
\end{eqnarray*}
where $\Delta$, $\nabla$ denote the Laplace and the gradient
operators on $M^{n}$ respectively, $\mathrm{div}$ is the divergence
operator on $M^{n}$, $\mathrm{Ric}$ is the Ricci tensor on $M^{n}$,
and $\mathrm{Hess}$ is the Hessian operator on $M^{n}$ associated to
the metric $\langle\cdot,\cdot\rangle$. For the $N$-dimensional
Bakry-\'{E}mery Ricci curvature $\mathrm{Ric}^{N}_{\phi}$, one knows
$N>n$ or $N=n$ if $\phi$ is a constant function. Specially, when
$N=\infty$, one can define the so-called $\infty$-dimensional
Bakry-\'{E}mery Ricci curvature (simply, Bakry-\'{E}mery Ricci
curvature or weighted Ricci curvature) as follows
 \begin{eqnarray*}
 \mathrm{Ric}_{\phi}=\mathrm{Ric}+\mathrm{Hess}\phi.
 \end{eqnarray*}
D. Bakry and M. \'{E}mery \cite{BE} introduced these curvature
tensors, and studied their relationship with diffusion processes.
For the domain $\Omega$ on
$(M^{n},\langle\cdot,\cdot\rangle,e^{-\phi}dv)$, one can also define
a notion \emph{weighted volume} (or $\phi$-volume) as
follows\footnote{~Generally, except specified, same mathematical
notations would have the same meanings in this paper.}
\begin{eqnarray*}
\mathrm{vol}_{\phi}(\Omega)=|\Omega|_{n,\phi}:=\int_{\Omega}d\mu=\int_{\Omega}e^{-\phi}dv,
\end{eqnarray*}
where $d\mu:=e^{-\phi}dv$ stands for the weighted volume density.
Clearly, if $\phi=const.$ is a constant function, the
Witten-Laplacian $\Delta_{\phi}$, Bakry-\'{E}mery Ricci curvature
$\mathrm{Ric}_{\phi}$, the weighted volume $|\Omega|_{n,\phi}$ would
degenerate into the usual Laplacian $\Delta$, the Ricci curvature
tensor $\mathrm{Ric}$, and the volume
$\mathrm{vol}(\Omega)=|\Omega|_{n}=\int_{\Omega}dv$ (i.e. the volume
of $\Omega$), respectively. We wish to refer readers to \cite{BQ,
LJ1, LV} for further details about geometric properties of the
Bakry-\'{E}mery Ricci curvature tensor.

 In the past four decades, smooth metric measure spaces have been extensively investigated in geometric analysis, and
 have close relation with probability theory and optimal transport.
We wish to mention again several important examples listed in our
previous work \cite[Section
 1]{ZLM} to
support
 this viewpoint --- G. Perelman's pioneering work \cite{PG} on the
 entropy formula of the Ricci flow under a weighted measure,
 Munteanu-Wang's series works \cite{MW1, MW2, MW3} about geometry of manifolds with
 weighted densities, Wei-Wylie's interesting work \cite{WgW} on some comparison
 theorems for manifolds with Bakry-\'{E}mery Ricci curvature
 bounded. Besides, recently, by mainly using the rearrangement technique and suitably constructing
trial functions, under the constraint of fixed weighted volume, R.
F. Chen and J. Mao \cite{CM} successfully obtain several
isoperimetric inequalities for the first and the second Dirichlet
eigenvalues, the first nonzero Neumann eigenvalue of the
Witten-Laplacian on bounded domains in space forms. These spectral
isoperimetric inequalities extend the classical ones (i.e., the
Faber-Krahn inequality, the Hong-Krahn-Szeg\H{o} inequality, and the
Szeg\H{o}-Weinberger inequality) of the Laplacian.

A smooth metric measure metric space
$(M^{n},\langle\cdot,\cdot\rangle,e^{-\phi}dv)$ is said to be
quasi-Einstein if
\begin{eqnarray*}
\mathrm{Ric}^{N}_{\phi}=\tau\langle\cdot,\cdot\rangle
\end{eqnarray*}
for some constant $\tau$. Specially, when $N=\infty$, the above
equality becomes
\begin{eqnarray*}
\mathrm{Ric}_{\phi}=\tau\langle\cdot,\cdot\rangle,
\end{eqnarray*}
and in this setting, $(M^{n},\langle\cdot,\cdot\rangle,e^{-\phi}dv)$
is called a gradient Ricci soliton. A gradient Ricci soliton is
called expanding, steady or shrinking if $\tau<0$, $\tau=0$, and
$\tau>0$, respectively. Ricci solitons are natural extension of
Einstein manifolds, and play an important role in the singularity
analysis of Ricci flow (see e.g. \cite{BR, PG}). We wish to refer
readers to a nice survey \cite{CHD} on the topic of Ricci solitons.

On the smooth metric measure metric space
$(M^{n},\langle\cdot,\cdot\rangle,e^{-\phi}dv)$, it is natural to
consider the following \emph{weighted heat equation}\footnote{~For
convenience, usually the partial derivative $\partial u/\partial t$
of $u(x,t)$ was rewritten as $u_{t}$. Similarly, $\partial
u/\partial x$, $\partial^{2}u/\partial x^2$ can be rewritten as
$u_{x}$ and $u_{xx}$, respectively. This convention on rewriting
partial derivatives of multivariable functions would also be used in
the sequel.}
\begin{eqnarray} \label{weighted-h}
\Delta_{\phi}u-\frac{\partial u}{\partial t}=0
\end{eqnarray}
instead of the classical heat equation, and correspondingly, the
\emph{weighted heat operator} $\mathcal{L}^{\phi}$ can be
well-defined as follows
\begin{eqnarray*}
\mathcal{L}^{\phi}:=\Delta_{\phi}-\frac{\partial}{\partial t}.
\end{eqnarray*}
Obviously, if $\phi=const.$, then the weighted heat equation
(\ref{weighted-h}), the weighted heat operator $\mathcal{L}^{\phi}$
degenerate into the heat equation $\Delta u-u_{t}=0$ and the heat
operator $\mathcal{L}:=\Delta-\partial/\partial t$, respectively.
Besides, if a solution $u$ to the equation (\ref{weighted-h}) is
independent of the time variable $t$, then $u$ is a weighted
harmonic function (also called $\phi$-harmonic function).

 We need the following notion:
\begin{defn}
Let $M^n$ be a Riemannian $n$-manifold and $\phi$ be a real-valued
smooth function on $M^n$. A fundamental solution\footnote{~It is
easy to see that if $u$ is a solution to the equation
(\ref{weighted-h}), then $-u$ is also. Hence, in some literatures
the positivity of the heat kernel has not been specially mentioned.
However, as a default, the kernel is usually required to be
positive. We shall not mention this point again in the sequel.},
which is called the weighted heat kernel, of the weighted heat
equation (\ref{weighted-h}) on $M^n$ is a continuous function
$H^{\phi}=H^{\phi}(x,y,t)$, defined on $M^{n}\times
M^{n}\times(0,\infty)$, which is $C^2$ with respect to $x$, $C^1$
with respect to $t$, and which satisfies
\begin{eqnarray*}
\mathcal{L}^{\phi}_{x}H^{\phi}=0,\qquad
\lim\limits_{t\rightarrow0}H^{\phi}(x,y,t)=\delta^{\phi}_{y}(x),
\end{eqnarray*}
where $\delta^{\phi}_{y}(x)$ is the weighted Dirac delta function,
that is, for all smooth functions $f$ on $M^n$ with compact support
(i.e. $f\in C^{\infty}_{0}(M^n)$), we have, for every $y\in M^n$,
\begin{eqnarray*}
\lim\limits_{t\rightarrow0}\int\limits_{M^n}f(x)H^{\phi}(x,y,t)e^{-\phi}dv(x)=f(y).
\end{eqnarray*}
\end{defn}

\begin{remark}
\rm{ In some literatures, see e.g. \cite{WW1, WW2}, the weighted
heat equation (\ref{weighted-h}), the weighted heat kernel are also
called $\phi$-heat equation and $\phi$-heat kernel, respectively.
Clearly, if $\phi=const.$, then the weighted heat kernel reduces
into the usual heat kernel of the heat equation.
 }
\end{remark}

 It
is well-known that the heat kernel of the heat equation in the
Euclidean $n$-space $\mathbb{R}^n$ is explicitly given as follows
\begin{eqnarray*}
H(x,y,t)=\frac{1}{(4\pi t)^{n/2}}\times
\exp\left(-\frac{|x-y|^2}{4t}\right).
\end{eqnarray*}
Now, we wish to recall several examples to intuitively show that
even in the one dimensional setting, the weighted heat kernel in
several nontrivial cases (i.e. $\phi$ is not a constant function) is
more complicated than the one in the trivial case, i.e.
$\phi=const.$ and $H^{\phi}(x,y,t)=H(x,y,t)=\frac{1}{(4\pi
t)^{1/2}}\times \exp\left(-\frac{|x-y|^2}{4t}\right)$ corresponding
to the trivial case.

\begin{example} [$\phi$-heat kernel for steady Gaussian soliton
\cite{WW1}] \label{exam-1} \rm{ Denote by
$(\mathbb{R},g_{E},e^{-\phi}dx)$ a $1$-dimensional steady Gaussian
soliton, where $g_{E}$ is the Euclidean metric and $\phi(x)=\pm x$.
Then $\mathrm{Ric}_{\phi}=0$,
$\Delta_{\phi}=\frac{d^2}{dx^{2}}\mp\frac{d}{dx}$, and by using the
separation of variables method, the weighted heat kernel is given by
\begin{eqnarray*}
H^{\phi}(x,y,t)=\frac{e^{\pm\frac{x+y}{2}}\cdot e^{-t/4}}{(4\pi
t)^{1/2}}\times \exp\left(-\frac{|x-y|^2}{4t}\right).
\end{eqnarray*}
For detailed calculations, see \cite[Appendix]{WW1}. }
\end{example}

\begin{example} [Mehler heat kernel for shrinking Gaussian soliton \cite{GA}]
\rm{Let $(\mathbb{R},g_{E},e^{-\phi}dx)$ be a $1$-dimensional
shrinking Gaussian soliton, where $g_{E}$ is the Euclidean metric
and $\phi(x)=x^{2}$. Then $\mathrm{Ric}_{\phi}=2$,
$\Delta_{\phi}=\frac{d^2}{dx^{2}}-2x\frac{d}{dx}$, and in this
setting the weighted heat kernel is given by
\begin{eqnarray*}
H^{\phi}(x,y,t)=\frac{1}{(2\pi\sinh2t)^{1/2}}\times\exp\left(\frac{2xye^{-2t}-(x^{2}+y^{2})e^{-4t}}{1-e^{-4t}}+t\right).
\end{eqnarray*}
}
\end{example}

\begin{example} [Mehler heat kernel for expanding Gaussian soliton \cite{GA}]
\rm{Let $(\mathbb{R},g_{E},e^{-\phi}dx)$ be a $1$-dimensional
expanding Gaussian soliton, where $g_{E}$ is the Euclidean metric
and $\phi(x)=-x^{2}$. Then $\mathrm{Ric}_{\phi}=-2$,
$\Delta_{\phi}=\frac{d^2}{dx^{2}}+2x\frac{d}{dx}$, and in this
setting the weighted heat kernel is given by
\begin{eqnarray*}
H^{\phi}(x,y,t)=\frac{1}{(2\pi\sinh2t)^{1/2}}\times\exp\left(\frac{2xye^{-2t}-(x^{2}+y^{2})}{1-e^{-4t}}-t\right).
\end{eqnarray*}
}
\end{example}

There are so many mathematicians having interest on the
(non-weighted or weighted) heat kernel and its applications, and
many interesting results have been obtained. Here we wish to give a
partial list of those results which influence me a lot. In 1981, S.
Y. Cheng, P. Li and S. T. Yau \cite{CLY} gave uniform Gaussian
estimates for the heat kernel on Riemannian manifolds with sectional
curvature bounded from below. Their estimates were improved by J.
Cheeger, M. Gromov, M. Taylor \cite{CGT} to manifolds with bounded
geometry. By using the gradient estimate technique and the Harnack
inequality, P. Li and S. T. Yau \cite{LY} obtained sharp Gaussian
upper and lower bounds for the Schr\"{o}dinger operator on
Riemannian manifolds with nonnegative Ricci curvature. For smooth
metric measure spaces with $\mathrm{Ric}^{N}_{\phi}$ ($N<\infty$)
bounded from below by a negative quadratic function, X. D. Li
\cite{LXD} got Gaussian estimates for the weighted heat kernel, and
proved an $L^{1}_{\phi}$-Liouville theorem. In \cite{WJY}, J. Y. Wu
obtained weighted heat kernel estimates on smooth metric measure
spaces with the weighted Ricci curvature bounded from below by a
negative constant and the potential function bounded. By using the
De Giorgi-Nash-Moser theory and the weighted version of Davies's
integral estimate \cite{DEB}, J. Y. Wu and P. Wu \cite[Theorem
1.1]{WW1} successfully gave a local Gaussian upper bound for the
weighted heat kernel on complete noncompact smooth metric measure
spaces $(M^{n},\langle\cdot,\cdot\rangle,e^{-\phi}dv)$ with
$\mathrm{Ric}_{\phi}\geq0$. Their upper bound estimate generalizes
the classical result of Li-Yau \cite{LY}. As applications of their
local Gaussian upper bound estimate for the weighted heat kernel,
they showed that if furthermore the potential function $\phi$ has at
most quadratic growth, then a sharp $L^{1}_{\phi}$-Liouville theorem
for $\phi$-subharmonic functions and an $L^{1}_{\phi}$-uniqueness
property for nonnegative solutions of the weighted heat equation can
be obtained (see \cite[Theorem 1.5]{WW1}). Besides, it is not hard
to see from \cite{CZ, H1} that any complete noncompact shrinking or
steady gradient Ricci soliton has nonnegative weighted Ricci
curvature and its potential function $\phi$ should have at most
quadratic growth. Then together with Wu-Wu's
$L^{1}_{\phi}$-Liouville theorem mentioned above, it is easy to know
that any $L^{1}_{\phi}$-integral $\phi$-subharmonic function on
gradient shrinking and steady Ricci solitons must be constant (see
\cite[Corollary 1.6]{WW1}). Later, as a sequel of the work
\cite{WW1}, by mainly using the De Giorgi-Nash-Moser theory, J. Y.
Wu and P. Wu \cite[Theorem 1.1]{WW2} obtained local Gaussian (upper
and lower) bounds for the weighted heat kernel on complete
noncompact smooth metric measure spaces
$(M^{n},\langle\cdot,\cdot\rangle,e^{-\phi}dv)$ with Bakry-\'{E}mery
Ricci curvature bounded from below by some nonpositive constant
(i.e. $\mathrm{Ric}_{\phi}\geq-(n-1)K$ for some constant $K\geq0$).
Moreover, when the curvature assumption was reduced to
$\mathrm{Ric}_{\phi}\geq0$ (i.e. $K=0$), they pointed out that the
local Gaussian lower bound for the weighted heat kernel can be
achieved by the $\phi$-heat kernel for steady Gaussian soliton
(explicitly given in Example \ref{exam-1}) as long as the time $t$
is large enough. Besides, as applications, they used their local
Gaussian (upper and lower) bounds for the weighted heat kernel to
obtain estimates for closed eigenvalues and for $\phi$-Green's
function of the Witten-Laplacian -- see \cite[Sections 6 and 7]{WW2}
for details. In our recent work \cite{ZLM}, by utilizing some
functional inequalities and combining with De Giorgi-Moser iteration
technique, some estimates for the weighted $L^{\infty}$-norm of
eigenfunctions of the Witten-Laplacian on smooth metric measure
spaces have been given. Besides, we also proved the regularity and
uniqueness of the weighted heat kernel on smooth metric measure
spaces, which generalized Chavel-Feldman's related results in
\cite{CF}.

In this paper, we wish to give a comparison theorem for the weighted
heat kernel (of the weighted heat equation) on complete manifolds
with radial curvatures bounded. In this setting, the model space
would be \emph{spherically symmetric manifolds}. In order to state
our main result clearly, we need to recall some preliminary
knowledge about the exponential mapping, the cut locus, the radial
direction, the injectivity radius, the radial (Ricci or sectional)
curvature having a bound given by a continuous function (see
Definitions \ref{def-5.2} and \ref{def-5.3}), spherically symmetric
manifolds (see Definition \ref{def-ssm}), and so on. However, in
order to control the length of this section, we wish to give
preliminaries in Section \ref{Sect2} of this paper. By the way,
these preliminaries can also be found in our previous works
\cite[Section 1]{fmi}, \cite[Section 2.1 of Chapter 2]{JM3},
\cite[Section 2]{JM4}.

 For a given Riemannian $n$-manifold $M^n$ with radial (Ricci or sectional) curvature bounded by a continuous function
$\kappa(s)$ of the distance parameter $s:=d_{M^{n}}(q,\cdot)$, $q\in
M^n$, where $d_{M^{n}}$ denotes the distance function on $M^n$, we
consider the following initial value problem
\begin{eqnarray} \label{ODE}
\left\{
\begin{array}{lll}
f''(s)+\kappa(s)f(s)=0,  \qquad & 0<s<l,\\
f'(0)=1,~f(0)=0,\\
f(s)>0, \qquad & 0<s<l,
\end{array}
\right.
\end{eqnarray}
which will be used to determine our model space
$[0,l)\times_{f}\mathbb{S}^{n-1}$ (i.e. spherically symmetric
manifolds, which are called \emph{generalized space forms} in
\cite{KK}). See Definition \ref{def-ssm} in Section \ref{Sect2} for
the precise statement such that a domain to be spherically
symmetric. Readers can check our previous work \cite[pp.
705--709]{fmi} for a brief introduction and an interesting spectral
asymptotic property of spherically symmetric manifolds. Obviously,
if in (\ref{ODE}), $\kappa(s)\equiv K$ for some constant $K$, then
\begin{eqnarray*}
f(s)=\left\{
\begin{array}{llll}
\frac{\sin\sqrt{K}s}{\sqrt{K}}, & \quad  l= \frac{\pi}{\sqrt{K}}
  & \quad K>0,\\
 s, &\quad l=+\infty & \quad K=0, \\
\frac{\sinh\sqrt{-K}s}{\sqrt{-K}}, & \quad l=+\infty  &\quad K<0,
\end{array}
\right.
\end{eqnarray*}
and correspondingly, the model space reduces to space forms.

We need to require:

\vspace{2mm}

\noindent \textbf{Property 1}. \emph{The potential function
$\phi:(0,l)\rightarrow\mathbb{R}$ is additionally required to be a
radial function with respect to the Riemannian distance parameter,
and moreover, its smoothness at the starting point (no matter on the
given complete manifold but also on the model space) should be
assured.}

\vspace{2mm}

Our main result is:

\begin{theorem} \label{maintheorem}
Given a complete Riemannian $n$-manifold $M^n$, $n\geq2$, and
assuming that the potential function $\phi$ satisfies
\textbf{Property 1}, we can obtain:

(1) if $M^n$ has a radial Ricci curvature lower bound
$(n-1)\kappa(s)$ w.r.t. some point $q\in M^n$, then, for
$r_{0}<\min\{\ell(q),l\}$, the inequality
\begin{eqnarray} \label{6-1-1-ex}
H^{\phi}(q,y,t)\geq
H^{\phi}_{-}\left(d_{M^{-}}\left(q^{-},z\right),t\right),
\end{eqnarray}
holds for all $(y,t)\in B(q,r_{0})\times(0,\infty)$ with
  $d_{M^n}(q,y)=d_{M^{-}}(q^{-},z)$ for any $z\in M^{-}$, where $\ell(q)$ is defined by\footnote{~In fact, if one checks the details in Section \ref{Sect2}, he or she would
  find that the quantity $\ell(q)$ is closely related with the first conjugate points of $q$ along different
  unit minimizing geodesics $\gamma_{\xi}(s)=\exp_{q}(s\xi)$, with $\xi\in S_{q}M^{n}\subset T_{q}M^n$.} (\ref{key-def1}),
   the model space is
  $M^{-}=:[0,l)\times_{f}\mathbb{S}^{n-1}$ with the base point $q^{-}$ and $f$ determined by
  (\ref{ODE}), $d_{M^{-}}$ denotes the distance function
  on $M^{-}$. Moreover, the equality in
  (\ref{6-1-1-ex}) holds at some $(y_{0},t_{0})\in
  B(q,r_{0})\times(0,\infty)$ if and only if $B(q,r_{0})$ is
  isometric to $\mathscr{B}_{n}(q^{-},r_{0})$. Here $B(q,r_{0})$ is
  the geodesic ball on $M^{n}$ with center $q$ and radius $r_0$,
  while $\mathscr{B}_{n}(q^{-},r_{0})$ is the geodesic ball on
  $M^{-}$ with center $q^{-}$ and the same radius $r_{0}$;

(2) if $M^n$ has a radial sectional curvature upper bound
$\kappa(s)$ w.r.t. $q\in M^n$, then, for
$r_{0}<\min\{\mathrm{inj}(q),l\}$, the inequality
\begin{eqnarray} \label{6-1-2-ex}
H^{\phi}(q,y,t)\leq
H^{\phi}_{+}\left(d_{M^{+}}\left(q^{+},z\right),t\right),
\end{eqnarray}
holds for all $(y,t)\in B(q,r_{0})\times(0,\infty)$ with
  $d_{M^n}(q,y)=d_{M^{+}}(q^{+},z)$ for any $z\in M^{+}$, where
  $\mathrm{inj}(q)$ defined by (\ref{inj-R}) is the injectivity radius at
  $q$,
  the model space is $M^{+}=:[0,l)\times_{f}\mathbb{S}^{n-1}$ with the base point $q^{+}$ and $f$ determined by
  (\ref{ODE}), $d_{M^{+}}$ denotes the distance function
  on $M^{+}$. Moreover, the equality in
  (\ref{6-1-2-ex}) holds at some $(y_{0},t_{0})\in
  B(q,r_{0})\times(0,\infty)$ if and only if $B(q,r_{0})$ is
  isometric to $\mathscr{B}_{n}(q^{+},r_{0})$. Here $\mathscr{B}_{n}(q^{+},r_{0})$ denotes the geodesic ball on
  $M^{+}$ with center $q^{+}$ and radius $r_{0}$.

(The boundary condition will either be Dirichlet or Neumann.)
\end{theorem}

\begin{example}  [$1$-dimensional model] \label{example4}
\rm{ Let $(\mathbb{R},g_{E})$ be the Euclidean $1$-space with the
Euclidean metric $g_{E}$. For a fixed point $x_{0}\in\mathbb{R}$,
set $r:=g_{E}(x_{0},y)=|y-x_{0}|$, which of course measures the
Euclidean distance between two points $x_{0}$, $y$ in $\mathbb{R}$.
If the potential function $\phi$ was chosen to be $\phi=\phi(r)=\pm
r$, which is radial with respect to $x_0$, then
$\mathrm{Ric}_{\phi}=0$,
$\Delta_{\phi}=\frac{d^2}{dr^{2}}\mp\frac{d}{dr}$, and by using the
separation of variables method (see Appendix A), the weighted heat
kernel is given by
\begin{eqnarray*}
H^{\phi}(x_{0},y,t)=H^{\phi}(r,t)&=&\frac{e^{\pm r}\cdot
e^{-t/4}}{(4\pi
t)^{1/2}}\times\exp\left(-\frac{r^2}{4t}\right) \\
&=&\frac{e^{\pm|y-x_{0}|}\cdot e^{-t/4}}{(4\pi
t)^{1/2}}\times\exp\left(-\frac{|y-x_{0}|^2}{4t}\right).
\end{eqnarray*}

}
\end{example}

\begin{remark} \label{remark1-1}
\rm{ In fact, from Example \ref{example4} above, it is not hard to
explicitly get the weighted heat kernel (with a prescribed Dirichlet
or Neumann boundary condition) on balls in $\mathbb{R}^n$ if the
potential function $\phi$ was chosen to be $\phi(r)=\pm r$. Of
course, in this situation, $\phi$ is radial with respect to the
center of the given Euclidean $n$-ball.}
\end{remark}

\begin{example} \label{example55}
\rm{ (see \cite[Subsection 4.6]{GA}) Let $\mathbb{H}^{3}$ be the
hyperbolic $3$-space with constant\footnote{~For the weighted heat
kernel $H^{\phi}$ shown in Example \ref{example55}, there is no
essential difference between $\mathbb{H}^3$ and the hyperbolic
$3$-space with negative constant sectional curvature not equal to
$-1$.} sectional curvature $-1$ and the related volume element $dv$.
For a fixed point $x_{0}\in\mathbb{H}^3$, set
$r:=d_{\mathbb{H}^3}(x_{0},y)$, which measures the hyperbolic
distance between two points $x_{0}$, $y$ in $\mathbb{H}^3$. If the
potential function $\phi$ was chosen to be
 \begin{eqnarray*}
\phi(r)=-2\log\left(\frac{r}{\sinh r}\right),
 \end{eqnarray*}
i.e.
 $e^{-\phi(r)}=\left(\frac{r}{\sinh r}\right)^{2}$,
then the weighted heat kernel is given by
\begin{eqnarray*}
H^{\phi}(x_{0},y,t)=H^{\phi}(r,t)&=&\frac{ e^{-t/4}}{(4\pi
t)^{3/2}}\frac{r}{\sinh r}\times\exp\left(-\frac{r^2}{4t}\right) \\
&=&\frac{e^{-t/4}}{(4\pi
t)^{3/2}}\frac{d_{\mathbb{H}^3}(x_{0},y)}{\sinh(d_{\mathbb{H}^3}(x_{0},y))}\times\exp\left(-\frac{d^{2}_{\mathbb{H}^3}(x_{0},y)}{4t}\right).
\end{eqnarray*}
 Readers can check \cite[Subsection 4.6]{GA} for the information on
 the weighted heat kernel $H^{\phi}$ in hyperbolic spaces with other
 dimensions.
}
\end{example}

\begin{remark}
\rm{ From Example \ref{example4} and Remark \ref{remark1-1}, one
easily knows that for a nontrivial potential function $\phi$, it
would also be more complicated if $\phi$ was additionally required
to be radial. From this aspect, it should be important and
meaningful to get possible comparisons for the weighted heat kernel
as what we have done in Theorem \ref{maintheorem}. }
\end{remark}

\begin{remark}
\rm{In (1) of Theorem \ref{maintheorem}, $\phi$ should be a radial
function of the Riemannian distance parameter
$s=d_{M^n}(q,\cdot)=d_{M^{-}}(q^{-},\cdot)$ with the starting point
$q$ on $M^n$ or the starting point $q^{-}$ on the model manifold
$M^{-}$. However, in (2) of Theorem \ref{maintheorem},  $\phi$
should be a radial function of
$s=d_{M^n}(q,\cdot)=d_{M^{+}}(q^{+},\cdot)$ with the starting point
$q$ on $M^n$ or the starting point $q^{+}$ on the model manifold
$M^{+}$. }
\end{remark}

\begin{remark}
\rm{ One might see that $M^{+}$ coincides with $M^{-}$ if the (lower
or upper) bound function $\kappa(s)$ for the radial Ricci curvature
or for the radial sectional curvature was the same, since in this
case one can get the same warping function $f(s)$ by solving the
initial value problem (\ref{ODE}).

}
\end{remark}

\begin{remark}
\rm{ It should be interesting to know:
 \begin{itemize}
\item \emph{When does the initial value problem (\ref{ODE}) have a unique solution on $(0,\infty)$?}
 \end{itemize}
That is to say, under what kind of assumptions, one has $l=\infty$.
This question is important because  if $l=\infty$, the weighted heat
kernel comparison (\ref{6-1-1-ex}) or (\ref{6-1-2-ex}) might sustain
continuously without any interruption. By the Sturm-Picone
comparison theorem, it is easy to know that if $\kappa(s)\leq0$,
then (\ref{ODE}) has a unique solution on $(0,\infty)$, i.e.
$l=\infty$. What about other situations? In \cite{BLM},
\cite[Chapter 2, Section 2.6]{JM3}, some sufficient conditions
involving $\kappa(s)$ have been given such that the initial value
problem (\ref{ODE}) has a unique solution on $(0,\infty)$. To the
best of our knowledge, so for, the sufficient and necessary
condition of the result $l=\infty$ for the initial value problem
(\ref{ODE}) has not been given yet.
 }
\end{remark}

\begin{corollary} \label{coro-1.1}
If $\phi=const.$, then the weighted heat kernel
$H^{\phi}(\cdot,\cdot,\cdot)$ becomes the usual heat kernel
$H(\cdot,\cdot,\cdot)$ of the heat equation, and correspondingly,
Theorem \ref{maintheorem} degenerates into our previous work
\cite[Theorem 6.6]{JM4} exactly.
\end{corollary}

\begin{remark}
\rm{ In \cite[Theorem 6.6]{JM4}, the author has given lower and
upper bounds for the heat kernel of geodesic balls on complete
manifolds with radial (Ricci and sectional) curvatures bounded. This
conclusion extends Cheeger-Yau's lower bound estimate \cite{CY} and
Debiard-Gaveau-Mazet's upper bound estimate \cite{DGM},
respectively. Hence, one can easily see from Corollary
\ref{coro-1.1} that our Theorem \ref{maintheorem} here improves the
heat kernel estimates in \cite{CY, DGM, JM4} a lot. }
\end{remark}

\begin{remark}
\rm{ As mentioned in \cite[Remark 6.7]{JM4}, the completeness of the
prescribed manifold $M^n$ is a little strong to get comparison
results for the heat kernel, since Cheeger-Yau \cite{CY} have shown
that if the injectivity radius at some point $q$ of a prescribed
manifold $M^n$ is bounded from below, then under the assumptions on
curvature therein, a lower bound can be given for the heat kernel of
geodesic balls on $M^n$. Maybe the completeness of $M^n$ is also a
little strong to get comparisons (\ref{6-1-1-ex}) and
(\ref{6-1-2-ex}) for the weighted heat kernel. However, here we wish
to assume that $M^n$ is complete, this is because that if $M^n$ is
complete, then for $B(q,r_0)\subseteq M^n$ with $r_0$ finite we can
always find optimally continuous bounds for the radial Ricci and
sectional curvatures with respect to $q$ (see (\ref{slb}) and
(\ref{sub})). Therefore, the assumption on the completeness of $M^n$
is feasible. }
\end{remark}

On $(M^{n},\langle\cdot,\cdot\rangle,e^{-\phi}dv)$, for a bounded
domain $\Omega\subset M^{n}$, it is natural to consider the
following Dirichlet eigenvalue problem of the Witten-Laplacian
\begin{eqnarray} \label{eigenp-w1}
\left\{
\begin{array}{ll}
\Delta_{\phi} u + \lambda u=0 \qquad  &\mathrm{in}~~\Omega,  \\[0.5mm]
u= 0 \qquad &\mathrm{on}~~\partial \Omega,
\end{array} \right.
\end{eqnarray}
with $\partial\Omega$ the boundary\footnote{~In general case, to
avoid a little bit boring discussion on the regularity of the
boundary $\partial\Omega$, which is obviously \emph{not} the point
mainly considered in this paper, one might assume that
$\partial\Omega$ is smooth for the Dirichlet eigenvalue problem of
the Witten-Laplacian. However, if one checks carefully the details
in this paper, he or she would find that weaker regularity
assumptions on $\partial\Omega$, such as $\partial\Omega$ is
piecewise smooth or Lipschitz continuous, can also make sure that
the spectral results of this paper are valid. Based on this reason,
in this paper, we do not make any explicit assumption for the
regularity of $\partial\Omega$ and we hope that one could understand
that a suitable regularity assumption for $\partial\Omega$ has been
used \emph{potentially} in our paper.} of the bounded domain
$\Omega$. For this problem, one has (see e.g. \cite[pp.
1247-1248]{CM}):
\begin{itemize}
\item The eigenvalue problem (\ref{eigenp-w1}) only has discrete
spectrum, and all the elements (i.e. eigenvalues) in this discrete
spectrum can be listed non-decreasingly as follows
 \begin{eqnarray} \label{sequen-1}
 0<\lambda_{1,\phi}(\Omega)<\lambda_{2,\phi}(\Omega)\leq\lambda_{3,\phi}(\Omega)\leq\cdots\uparrow+\infty.
 \end{eqnarray}
Each eigenvalue\footnote{~For convenience and without any confusion,
except specification we wish to write $\lambda_{i,\phi}(\Omega)$ as
$\lambda_{i,\phi}$ directly. This abbreviation convention would also
be used for eigenvalues of other types discussed in the sequel.}
$\lambda_{i,\phi}$, $i=1,2,\ldots$, in the sequence (\ref{sequen-1})
was repeated according to its multiplicity (which is finite and
equals the dimension of the eigenspace of $\lambda_{i,\phi}$).

 \item By the variational principle, it is not hard to see that the
 first Dirichlet eigenvalue $\lambda_{1,\phi}(\Omega)$ can be
 characterized as follows
  \begin{eqnarray} \label{chr-1}
 \lambda_{1,\phi}(\Omega)=\inf\left\{\frac{\int_{\Omega}|\nabla f|^{2}e^{-\phi}dv}{\int_{\Omega}f^{2}e^{-\phi}dv}=
 \frac{\int_{\Omega}|\nabla f|^{2}d\mu}{\int_{\Omega}f^{2}d\mu}\Bigg{|}f\in
 W^{1,2}_{0,\phi}(\Omega),f\neq0\right\}.
  \end{eqnarray}
Using notations as in \cite{CM}, here $W^{1,2}_{0,\phi}(\Omega)$
denotes a Sobolev space, which is the completion of the set of
smooth functions (with compact support) $C^{\infty}_{0}(\Omega)$
under the following Sobolev norm
\begin{eqnarray*}
\|f\|^{\phi}_{1,2}:=\left(\int_{\Omega}f^{2}e^{-\phi}dv+\int_{\Omega}|\nabla
f|^{2}e^{-\phi}dv\right)^{1/2}.
\end{eqnarray*}
Moreover, the $k$-th Dirichlet eigenvalue $\lambda_{k,\phi}(\Omega)$
can be characterized as follows
 \begin{eqnarray*}
 \lambda_{k,\phi}(\Omega)=\inf\left\{\frac{\int_{\Omega}|\nabla f|^{2}d\mu}{\int_{\Omega}f^{2}d\mu}
 \Bigg{|}f\in W^{1,2}_{0,\phi}(\Omega),f\neq0,\int_{\Omega}ff_{i}d\mu=0\right\},
 \end{eqnarray*}
where $f_{i}$, $i=1,2,\cdots,k-1$, denotes an eigenfunction of
$\lambda_{i,\phi}(\Omega)$.
\end{itemize}
In our previous work \cite[Introduction]{CM1}, we have given a
detailed explanation on why is interesting and important to study
spectral geometric problems related to the Witten-Laplacian. We
already have some interesting works about spectral isoperimetric
inequalities, spectral estimates and geometric functional
inequalities related to the Witten-Laplacian -- see, e.g., \cite{CM,
CM1, DMWW, JM1, JM2, YWMD}.

By the way, if $\phi=const.$, then the eigenvalue problem
(\ref{eigenp-w1}) would degenerate into the classical fixed membrane
problem of the Laplacian, which has been extensively studied and can
be used as a mathematical model to describe the vibration of a
membrane with boundary fixed (when $M^{2}=\mathbb{R}^2$, and
moreover, in this case, $\Omega$ is a bounded planar domain). We
wish to refer readers to several books \cite{IC, MDW, UH} for a
relatively comprehensive knowledge about the history and some
progresses of spectral problems of the Laplacian.

More than 10 years ago,  P. Freitas, J. Mao, and I. Salavessa
\cite{fmi} successfully extended the classical Bishop's volume
comparison theorems to a more general setting (see also Section
\ref{Sect2} for details), that is, the curvature assumptions for
volume comparisons have been weakened from (Ricci or sectional)
curvature bounded (by a constant) to radial (Ricci or sectional)
curvature bounded (by some continuous function with respect to the
Riemannian distance parameter), and then they improved S. Y. Cheng's
eigenvalue comparison theorems in \cite{CSY1, CSY2} by applying
their volume comparisons and using spherically symmetric manifolds
as the model space.

As applications of the weighted heat kernel comparison (i.e. Theorem
\ref{maintheorem}), we can get the following eigenvalue comparison
theorems of the Witten-Laplacian, which can be seen as the weighted
version of those shown in \cite[Theorems 3.6 and 4.4]{fmi}.

\begin{theorem} \label{theo-1}
Assume that an $n$-dimensional ($n\geq2$) complete Riemannian
manifold $M^n$ has a radial Ricci curvature lower bound
$(n-1)\kappa(s)$ w.r.t. $q\in M^n$, where $s=d_{M^n}(q,\cdot)$
denotes the Riemannian distance from $q$. Assume that the potential
function $\phi$ has \textbf{Property 1}. Then  for
$r_{0}<\min\{\ell(q),l\}$ with $\ell(q)$ defined as in
(\ref{key-def1}), one has
\begin{eqnarray} \label{ECT-1}
\lambda_{1,\phi}(B(q,r_0))\leq\lambda_{1,\phi}\left(\mathscr{B}_{n}(q^{-},r_0)\right),
\end{eqnarray}
and the equality holds if and only if $B(q,r_0)$ is isometric to
$\mathscr{B}_{n}(q^{-},r_0)$.
\end{theorem}

\begin{theorem}  \label{theo-2}
Assume that an $n$-dimensional ($n\geq2$) complete Riemannian
manifold $M^n$ has a radial sectional curvature upper bound
$\kappa(s)$ w.r.t. $q\in M^n$, and assume that the potential
function $\phi$ has \textbf{Property 1}. Then for
$r_{0}<\min\{\mathrm{inj}(q),l\}$ with $\mathrm{inj}(q)$ defined as
in (\ref{inj-R}), one has
\begin{eqnarray}  \label{ECT-2}
\lambda_{1,\phi}(B(q,r_0))\geq\lambda_{1,\phi}\left(\mathscr{B}_{n}(q^{+},r_0)\right),
\end{eqnarray}
and the equality holds if and only if $B(q,r_0)$ is isometric to
$\mathscr{B}_{n}(q^{+},r_0)$.
\end{theorem}

\begin{remark}
\rm{ Denote by $\lambda_{k}(\cdot)$ the $k$-th Dirichlet eigenvalue
of the Laplacian on a given geometric object. Clearly, if
$\phi=const.$, then (\ref{ECT-1})-(\ref{ECT-2}) reduce to
\begin{eqnarray} \label{ECT-3}
\lambda_{1}(B(q,r_0))\leq\lambda_{1}\left(\mathscr{B}_{n}(q^{-},r_0)\right)
\end{eqnarray}
and
\begin{eqnarray} \label{ECT-4}
\lambda_{1}(B(q,r_0))\geq\lambda_{1}\left(\mathscr{B}_{n}(q^{+},r_0)\right),
\end{eqnarray}
respectively. Moreover, when the equalities in the above two
eigenvalue inequalities hold, the corresponding isometric rigidity
between two geodesic balls can also be attained. This conclusion is
exactly the assertions in \cite[Theorems 3.6 and 4.4]{fmi}. }
\end{remark}

\begin{remark} \label{remark-1.19}
\rm{ In Appendix B, another approach of proving Theorem \ref{theo-1}
and Theorem \ref{theo-2} would be shown carefully. All the
calculations in Appendix B were carried out by my master student J.
J. Gu under my guidance, and would constitute the main content of
her master degree's dissertation \cite{GJJ} which will be formally
issued in June 2026. By the way, the assertion in Theorem
\ref{theo-1} will be generalized to the case of nonlinear weighted
$p$-Laplacian ($1<p<\infty$) in Appendix B as well.
 }
\end{remark}

\begin{remark}
\rm{ Specially, if curvature assumptions were strengthened to be
$\mathrm{Ric}(M^n)\geq(n-1)K$ or $\mathrm{Sec}(M^n)\leq K$ for some
constant $K$ (i.e. the Ricci (or sectional) curvature of the
complete $n$-manifold $M^n$ is bounded from below (or above) by some
constant), then the eigenvalue comparisons
(\ref{ECT-3})-(\ref{ECT-4}) would degenerate into
\begin{eqnarray} \label{ECT-5}
\lambda_{1}(B(q,r_0))\leq\lambda_{1}\left(\mathcal{B}_{n}(K,r_0)\right)
\end{eqnarray}
and
\begin{eqnarray} \label{ECT-6}
\lambda_{1}(B(q,r_0))\geq\lambda_{1}\left(\mathcal{B}_{n}(K,r_0)\right),
\end{eqnarray}
where $\mathcal{B}_{n}(K,r_0)$ denotes a geodesic ball in the
$n$-dimensional space form $\mathbb{M}^{n}(K)$ having constant
sectional curvature $K$. Moreover, the equality in (\ref{ECT-5}) or
(\ref{ECT-6}) holds if and only if $B(q,r_0)$ is isometric to
$\mathcal{B}_{n}(K,r_0)$. Since any space form is two-points
homogeneous, there is no need to mention any information for the
center of the geodesic ball $\mathcal{B}_{n}(K,r_0)$. Eigenvalue
comparisons (\ref{ECT-5})-(\ref{ECT-6}) and related rigidity results
are main conclusions in S. Y. Cheng's works \cite{CSY1, CSY2}. Based
on this reason, we wish to call (\ref{ECT-1})-(\ref{ECT-4})
``\emph{Cheng-type eigenvalue comparisons}". In fact, this naming
convention has already been used in our previous works \cite{fmi,
JM3, JM4}.

 }
\end{remark}

\begin{remark}
\rm{ In \cite[Section 4]{SAG}, by mainly using a similar argument to
that in Cheng's proof of \cite[Theorem 1.1]{CSY1}, A. G. Setti
obtained an eigenvalue comparison theorem for the first Dirichlet
eigenvalue of the Witten-Laplacian on geodesic balls of complete
Riemannian $n$-manifolds with a modified Ricci curvature bounded
from below by a constant. More precisely, for a complete Riemannian
$n$-manifold $M^{n}$ with the metric $\langle\cdot,\cdot\rangle$,
let $w$ be a given smooth strictly positive function on $M^{n}$
(which was called \emph{weight function} therein), and then A. G.
Setti defined \emph{a modified Ricci curvature} as follows
 \begin{eqnarray} \label{MRC}
 \mathrm{Ric}_{w}:=\mathrm{Ric}-w^{-1}\mathrm{Hess}w.
 \end{eqnarray}
Since $w$ is a smooth strictly positive function on $M^{n}$, it is
easy to know that there must exist a potential function $\phi$
(which is smooth real-valued) on $M^{n}$ such that $w=e^{-\phi}$,
and then one has from (\ref{MRC}) that
 \begin{eqnarray} \label{MRC-1}
 \mathrm{Ric}_{w}=\mathrm{Ric}-e^{\phi}\mathrm{Hess}(e^{-\phi})=\mathrm{Ric}+\mathrm{Hess}\phi-d\phi\otimes
d\phi=
 \mathrm{Ric}_{\phi}-d\phi\otimes
d\phi.
 \end{eqnarray}
 Therefore, it follows from (\ref{MRC-1}) that comparing with the
 Bakry--\'{E}mery Ricci curvature tensor $\mathrm{Ric}_{\phi}$,
the modified Ricci
 curvature tensor $\mathrm{Ric}_{w}$ defined in \cite{SAG} has an extra term $-d\phi\otimes
d\phi$. Besides, the only difference between $\mathrm{Ric}_{w}$ and
the $N$-dimensional Bakry-\'{E}mery Ricci curvature tensor
$\mathrm{Ric}_{\phi}^{N}$ is the coefficient $\frac{1}{N-n}$ in
front of the term $-d\phi\otimes d\phi$, which in my opinion is not
an essential difference. In \cite[Theorem 4.2]{SAG}, A. G. Setti
showed that if
 \begin{eqnarray*}
 \mathrm{Ric}_{w}(v_{x},v_{x})\geq\alpha
 \end{eqnarray*}
for some constant $\alpha$ in the geodesic ball $B(q,R)$ on a given
complete Riemannian $n$-manifold $M^n$, where $\forall x\in B(q,R)$
and $v_{x}$ is radial unit tangent vector at $x$ defined as
(\ref{radial-V}), then
 \begin{eqnarray} \label{1-16}
\lambda_{1,\phi}(B(q,R))\leq
\lambda_{1}\left(\mathcal{B}_{n+1}(\alpha/n,R)\right)
 \end{eqnarray}
 with the weight function $w$ satisfying $w=e^{-\phi}$. That is to
 say, the first Dirichlet eigenvalue $\lambda_{1,\phi}(B(q,R))$ of
 the Witten-Laplacian $\Delta_{\phi}$ on the geodesic ball $B(q,R)\subseteq M^{n}$
 is bounded from above by the first Dirichlet eigenvalue $\lambda_{1}\left(\mathcal{B}_{n+1}(\alpha/n,R)\right)$ of the
 Laplacian on a geodesic ball $\mathcal{B}_{n+1}(\alpha/n,R)$ in the
 $(n+1)$-dimensional space form $\mathbb{M}^{n+1}(\alpha/n)$.
 Specially, in \cite[Section 4]{SAG}, he also mentioned that if
 furthermore
 $\alpha>0$, then using a Bishop-type volume comparison theorem
 \cite[Theorem 4.1]{SAG}, one knows that
 $d_{M^n}(q,\hat{q})\leq\pi\sqrt{\frac{n}{\alpha}}$, $\forall\hat{q}\in
 M^{n}$. Hence, if $R>\pi\sqrt{\frac{n}{\alpha}}$, then
 $B(q,R)=M^{n}$ and
 $\mathcal{B}_{n+1}(\alpha/n,R)=\mathbb{M}^{n+1}(\alpha/n)$, which
 implies $\lambda_{1,\phi}(B(q,R))=0=
\lambda_{1}\left(\mathcal{B}_{n+1}(\alpha/n,R)\right)$, and the
comparison (\ref{1-16}) holds trivially. That is to say, the
eigenvalue comparison (\ref{1-16}) does not have the rigidity
characterization for the equality ``$=$", however in some special
cases the equality can be achieved.

Recently, for a complete Riemannian $n$-manifold $M^{n}$ and a point
$q\in M^{n}$, if the modified Ricci curvature $\mathrm{Ric}_{w}$
satisfies
\begin{eqnarray*}
\mathrm{Ric}_{w}(v_{x},v_{x})\geq\kappa(r(x)),
\end{eqnarray*}
where $\forall x\in B(q,R)$, $r(x)=d_{M^{n}}(q,x)$, and
$\kappa:(0,R)\rightarrow \mathbb{R}$ is a continuous function of the
radial distance parameter, then we \cite{CMM} can improve Setti's
eigenvalue comparison (\ref{1-16}) to a more general setting, and
the corresponding model space becomes $(n+1)$-dimensional
spherically symmetric manifolds.
  }
\end{remark}

\begin{remark}
\rm{ The reason we do not put also the eigenvalue comparison results
\cite{CMM} in this paper is that (\ref{ECT-1})-(\ref{ECT-2}) are the
comparisons between eigenvalues of the same type (i.e. the first
Dirichlet eigenvalues of the Witten-Laplacian on different geodesic
balls), while the ones in \cite{CMM} are comparisons between
Dirichlet eigenvalues of different elliptic differential operators.
}
\end{remark}

\begin{remark}
\rm{ The study of extremum problems is one of core topics in
Spectral Geometry, and we wish to refer readers to A. Henrot's book
\cite{HA} for a relatively comprehensive knowledge about this topic.
As briefly mentioned before, for an $n$-dimensional($n\geq2$)
Riemannian manifold  $M^{n}$ and a bounded domain $\Omega$ on $
M^n$, if the potential function $\phi$ was assumed to be radial,
suitably concave or convex, then recently R. F. Chen and J. Mao
\cite{CM} successfully and completely solved the extremum problem
\begin{eqnarray*}
 \min\{\lambda_{k,\phi}(\Omega)|\Omega\subset M^{n}, |\Omega|_{n,\phi}=const.\}
\end{eqnarray*}
 when $k=1, 2$ and $M^{n}\equiv\mathbb{M}^{n}(K)$, i.e. the
$n$-dimensional space form with constant sectional curvature $K$.
This conclusion definitely extends those classical spectral
isoperimetric inequalities, i.e. the Faber-Krahn inequality, the
Hong-Krahn-Szeg\H{o} inequality of the Laplacian under the
constraint of fixed volume. From the viewpoint of extremum problems,
our Theorems \ref{theo-1}-\ref{theo-2} here give characterizations
for the extremum value of the functional
\begin{eqnarray*}
B(q,r_0) \mapsto\lambda_{1,\phi}(B(q,r_0))
\end{eqnarray*}
provided the potential function $\phi$ satisfying \textbf{Property
1}, where $B(q,r_0)$ is the geodesic ball with center $q$ and fixed
radius $r_0$ on a given complete Riemannian $n$-manifold. Speaking
in other words, if the potential function $\phi$ satisfies
\textbf{Property 1}, then the value of the first Dirichlet
eigenvalue $\lambda_{1,\phi}(B(q,r_0))$ of the Witten-Laplacian
would change as long as the radial (Ricci or sectional) curvature
with respect to $q$ changes.

 Comparing with our research experience in
\cite{CM}, it is so surprising that except the radial curvature
assumptions, we can get the Cheng-type eigenvalue comparisons in
Theorems \ref{theo-1}-\ref{theo-2} only assuming that the potential
function $\phi$ satisfies \textbf{Property 1}.
 }
\end{remark}

\begin{remark}
\rm{ Specially, when $n=2$, $M^2$ is a complete surface, and as
mentioned in Section \ref{Sect2} below, for any $q\in M^2$, the
radial Ricci and sectional curvatures become exactly the Gaussian
curvature, and one can always find sharp lower and upper bounds
$\kappa_{\pm}(q,s)$ for the curvature assumptions with respect to
the given point $q$, and then it follows from Theorems
\ref{theo-1}-\ref{theo-2} that the functional
 \begin{eqnarray*}
 B(q,s) \mapsto\lambda_{1,\phi}(B(q,s))
 \end{eqnarray*}
would attain its minimal value if the Gaussian curvature in the
geodesic disk $B(q,s)$ vanishes identically, where $s\in(0,\infty)$
and the potential function $\phi$ satisfies \textbf{Property 1}. }
\end{remark}

This paper is organized as follows. Some preliminaries and useful
facts will be introduced in Section \ref{Sect2}. A proof of Theorem
\ref{maintheorem} (i.e. a weighted heat kernel comparison theorem
for the weighted heat equation on complete manifolds with radial
curvatures bounded), together with some basic properties related to
the weighted heat kernel, will be given in Section \ref{Sect3}. As
an application, in Section \ref{Sect4}  we will mainly use Theorem
\ref{maintheorem} to give a proof for two Cheng-type eigenvalue
comparisons (\ref{ECT-1}) in Theorem \ref{theo-1} and (\ref{ECT-2})
in Theorem \ref{theo-2}. The detailed calculations for the weighted
heat kernel given in Example \ref{example4} will be given in
Appendix A. We shall give another proof for eigenvalue comparisons
(\ref{ECT-1})-(\ref{ECT-2}) and their rigidity characterizations
(i.e. Theorems \ref{theo-1} and \ref{theo-2}) in Appendix B.

On $1^{\mathrm{st}}$ March 2026, we posted the first version of this
paper on arXiv. However, on $12^{\mathrm{th}}$ March 2026, we
noticed the work \cite{SPHC} which was just posted on arXiv on
$9^{\mathrm{th}}$ March 2026 (i.e. almost three days ago), and has
cited my two works \cite{DM, JM4}. We found that the assertion in
Theorem \ref{theo-2} can be improved to the case of weighted
$p$-Laplacian by making some adjustments to the argument in the
proof of \cite[Theorem 1.4]{SPHC}. We wish to write this content as
an Appendix C in the second version of this paper.

\section{Preliminaries and some useful facts} \label{Sect2}
\renewcommand{\thesection}{\arabic{section}}
\renewcommand{\theequation}{\thesection.\arabic{equation}}
\setcounter{equation}{0}

Given an $n$-dimensional ($n\geq2$) complete Riemannian manifold
$M^{n}$ with the metric $\langle\cdot,\cdot\rangle$, for any point
$q\in M^{n}$, the exponential map $\exp_q:\mathcal{D}_p \to
M^{n}\setminus Cut(q)$ gives a diffeomorphism from a star-shaped
open set $\mathcal{D}_{p}$ of the tangent space $T_{q}M^{n}$ with
\begin{eqnarray*}
\mathcal{D}_{q}=\left\{s\xi|~0\leq{s}<d_{\xi},~\xi\in{S^{n-1}_{q}}\right\}
\end{eqnarray*}
 onto the open set  $M^{n}\setminus Cut(q)$,
where $Cut(q)$ is the cut locus  of  $q$,  $S^{n-1}_{q}$ denotes the
unit sphere of $T_{q}M^n$, and $d_{\xi}$ is
 defined by
\begin{eqnarray} \label{add-extra1}
&&d_{\xi}=d_{\xi}(q) : =
\sup\{r>0|~\gamma_{\xi}(s)=\gamma_{(q,\xi)}(s):=
\exp_q(s\xi)~{\rm{is~ the~ unique}} \nonumber\\
&&\qquad\qquad\qquad \qquad\qquad\qquad {\rm{minimal
~geodesic~joining}}~q ~ {\rm{and}}~\gamma_{\xi}(r)\}.
\end{eqnarray}
For a fixed unit vector $\xi\in S^{n-1}_{q}\subset T_{q}M^n$, let
$\xi^{\bot}$ be the orthogonal complement of $\{\mathbb{R}\xi\}$ in
$T_{q}M^n$ and $\tau_{s}:
T_{q}M^n\rightarrow{T_{\exp_{q}(s\xi)}M^n}$ be the parallel
translation along $\gamma_{\xi}$. The path of linear transformations
$\mathbb{A}(s,\xi):\xi^{\bot}\rightarrow{\xi^{\bot}}$ is given by
 \begin{eqnarray*}
\mathbb{A}(s,\xi)\eta=(\tau_{s})^{-1}Y_{\eta}(s),
 \end{eqnarray*}
where $Y_{\eta}(s)=d(\exp_q)_{(s\xi)}(s\eta)$ is the Jacobi field
along $\gamma_{\xi}$ satisfying $Y_{\eta}(0)=0$, and
$(\nabla_{s}Y_{\eta})(0)=\eta$. This operator satisfies the Jacobi
equation
 $\mathbb{A}''+\mathcal{R}\mathbb{A}=0$ with initial conditions
 $\mathbb{A}(0,\xi)=0$, $\mathbb{A}'(0,\xi)=I$, where
$\mathcal{R}(s)$ is the self-adjoint operator on $\xi^{\bot}$, $
\mathcal{R}(s)\eta=(\tau_{s})^{-1}R(\gamma'_{\xi}(s),\tau_{s}\eta)
\gamma'_{\xi}(s)$ with $R$ the curvature tensor on $M^n$. The trace
of the later operator is just
 the radial Ricci tensor
$$\mathrm{Ric}_{\gamma_{\xi}(s)}
(\gamma'_{\xi}(s), \gamma'_{\xi}(s))$$
 along unit speed geodesics starting from
 $q$.
Gauss's lemma tells us that the
 Riemannian metric on $M^{n}\setminus Cut(q)$ in
the spherical geodesic coordinate chart can be expressed as
\begin{eqnarray} \label{5-1-1}
g:=
\langle\cdot,\cdot\rangle_{\exp_{q}(s\xi)}=ds^{2}+|\mathbb{A}(s,\xi)d\xi|^{2},
\quad \forall ~s\xi \in\mathcal{D}_{q}.
 \end{eqnarray}
We consider  the metric components  $g_{ij}(s,\xi)$, $i,j\geq 1$, in
a coordinate system $\{s, \xi_a\}_{a\geq 2}$ formed by fixing  an
orthonormal basis  $\{\eta_a\}_{a\geq 2}$ of
 $\xi^{\bot}=T_{\xi}S^{n-1}_q$, and extending it to a local frame $\{\xi_{a}\}_{a\geq 2}$ of
$S_q^{n-1}$. Define  a function $J>0$  on $\mathcal{D}_{q}$ by
\begin{equation} \label{J}
J^{n-1}=\sqrt{|g|}:=\sqrt{\det[g_{ij}]},
\end{equation}
that is, $\sqrt{|g|}=\det\mathbb{A}(t,\xi)$, and
 $dv=J^{n-1}dsd \sigma$ is the volume element of $M^{n}\setminus Cut(q)$,
where $d\sigma$ denotes the $(n-1)$-dimensional volume element on
$\mathbb{S}^{n-1}\equiv S_{q}^{n-1}\subseteq{T_{q}M^{n}}$. Since
$\tau_{s}: T_{q}M^n\rightarrow{T_{\exp_{q}(s\xi)}M^n}$ is an
isometry, one has
\begin{eqnarray*}
\langle
d(\exp_q)_{(s\xi)}(s\eta_{a}),d(\exp_q)_{(s\xi)}(s\eta_{b})\rangle=\langle\mathbb{A}(s,\xi)(\eta_{a}),\mathbb{A}(s,\xi)(\eta_{b})\rangle,
\end{eqnarray*}
and so $ \sqrt{|g|}=\det\mathbb{A}(s,\xi)$. Therefore, by applying
(\ref{5-1-1}) and (\ref{J}), it follows that the volume
$\mathrm{vol}(B(q,r))$ of the geodesic ball $B(q,r)$ on $M^{n}$ is
given by
\begin{eqnarray} \label{volume-a1}
\mathrm{vol}(B(q,r))= \int_{S_{q}^{n-1}}\int_0^{\min\{r,
d_{\xi}\}}\sqrt{|g|}dsd\sigma=\int_{S_{q}^{n-1}}\left(\int_0^{\min\{r,
d_{\xi}\}} \det(\mathbb{A}(s,\xi)) ds\right) d\sigma.
\end{eqnarray}
 Let
 \begin{eqnarray} \label{inj-R}
\mathrm{inj}(q):=d_{M^n}(q,Cut(q))=\min\limits_{\xi\in
S_{q}M^{n}}d_{\xi}
 \end{eqnarray}
 be the injectivity radius at $q$. Generally, one has $B(q,\mathrm{inj}(q))\subseteq
 M^{n}\setminus\mathrm{Cut}(q)$. Besides, for $r<\mathrm{inj}(q)$,
 one has
 \begin{eqnarray*}
\mathrm{vol}(B(q,r))=\int_{0}^{r}
\int_{S_{q}^{n-1}}\det(\mathbb{A}(s,\xi))d\sigma ds.
 \end{eqnarray*}
If $r(x)=d_{M^n}(x,q)$ denotes the intrinsic distance to the point
$q$, $x\in M^{n} \setminus (Cut(q)\cup \{q\})$, then the unit vector
field
 \begin{eqnarray} \label{radial-V}
v_x=\nabla r(x)
 \end{eqnarray}
is the  radial unit tangent vector at $x$, according to the
definition given in \cite{KK}. This is because for any $\xi\in
S_{q}^{n-1}$ and $s_{0}>0$, one has $\nabla
r(\gamma_{\xi}(s_0))=\gamma'_{\xi}(s_0)$ when the point
$\gamma_{\xi}(s_0)=\exp_{q}(s_{0}\xi)$ is away from the cut locus of
$q$ (see e.g. \cite{GrA}). Set
\begin{eqnarray} \label{key-def1}
\ell(q):=\sup\limits_{x\in M^n}r(x).
\end{eqnarray}
One has $\ell(q)=\max_{\xi\in S_{q}M^{n}}d_{\xi}$ (see page 704 of
\cite[Section 2]{fmi} for the explanation). Clearly,
 \begin{eqnarray*}
\ell(q)\geq\mathrm{inj}(q).
\end{eqnarray*}
We also recall the following  inequality  about $r(x)$ (cf.
\cite{PP}, Prop.\ 39, and pp.\ 266-267), with $\partial_r=\nabla r$
as a vector of differentiation (see Prop.\ 7 on p.\ 47  of the same
reference),
$$
\partial_r\Delta r \leq \partial_r\Delta r+\|\mathrm{Hess}\, r\|^2
=-\mathrm{Ric}(\partial_r,
\partial_r),\quad\quad\mbox{with}\quad \Delta
r=\partial_r\ln(\sqrt{|g|}),
$$
which implies that
\begin{eqnarray}
&&
J''+\frac{1}{(n-1)}\,\mathrm{Ric}{(\gamma_{\xi}'(t),\gamma_{\xi}'(t))}
\, J\leq 0,  \label{J''}\\
&&J(0,\xi)=0,\quad J'(0,\xi)=1 \label{J0}.
\end{eqnarray}
(\ref{J''}) and (\ref{J0}) make a fundamental role in the derivation
of the so-called \emph{generalized Bishop's volume comparison
theorem I} (see Theorem \ref{BishopI} below), which has been proven
in our previous works \cite{fmi, JM3}.

We need the following concept:

\begin{defn} \label{def-ssm}
A domain $\Omega=\exp_q([0,l)\times{S}_q^{n-1}) \subset
M^{n}\backslash Cut(q)$, with $l<inj(p)$,   is said to be
spherically symmetric with respect to a point $q\in \Omega$, if
 and only if
the matrix $\mathbb{A}(s,\xi)$ satisfies $\mathbb{A}(s,\xi)=f(s)I$,
for a function $f\in{C^{2}([0,l))}$,  with   $f(0)=0$, $f'(0)=1$,
and  $f|_{(0,l)}>0$.
\end{defn}

Obviously, on the domain $\Omega$ given as in Definition
\ref{def-ssm} the Riemannian metric of $M^n$ can be expressed by
\begin{eqnarray} \label{metric-ssm}
 \langle\cdot,\cdot\rangle=ds^{2}+f^{2}(s)|d\xi|^{2}, \qquad \forall
~\xi \in S_{q}M^{n},~0\leq s<l,
 \end{eqnarray}
where $|d\xi|^{2}$ denotes the round metric on the unit sphere
$\mathbb{S}^{n-1}\subseteq\mathbb{R}^n$. A standard model for
spherically symmetric manifolds is given by the quotient manifold of
the warped product $[0,l)\times_{f}\mathbb{S}^{n-1}$ equipped with
the metric (\ref{metric-ssm}), where the warping function $f$
satisfies the conditions in Definition \ref{def-ssm}, and all pairs
$(0,\xi)$ are identified with a single point $q$ (see \cite{BB}). We
wish to refer readers to e.g. \cite{MDW, PP} for the notion and some
properties of warped product manifolds. In fact, an $n$-dimensional
spherically symmetric manifold $M^{\ast}$ satisfying those
conditions in Definition \ref{def-ssm} is actually a quotient space
$M^{\ast}=\left([0,l)\times_{f(s)}\mathbb{S}^{n-1}\right)/\thicksim$
with the equivalent relation ``$\thicksim$" given by
 \begin{eqnarray*}
(s,\xi)\thicksim(t,\eta)\Longleftrightarrow \left\{
\begin{array}{ll}
s=t~\mathrm{and}~\xi=\eta,~~~\mathrm{or}\\
s=t=0.
\end{array}
\right.
 \end{eqnarray*}
This equivalent relation is natural, and usually one can just use
$[0,l)\times_{f(s)}\mathbb{S}^{n-1}$ to represent this quotient
space. That is to say, $M^{\ast}=[0,l)\times_{f(s)}\mathbb{S}^{n-1}$
with $f$ satisfying conditions in Definition \ref{def-ssm} is a
spherically symmetric manifold with
$q=\{0\}\times_{f}\mathbb{S}^{n-1}$ the base point and
(\ref{metric-ssm}) as its metric. For $M^{\ast}$ and $r<l$, by
(\ref{volume-a1}) it follows that the volume of the geodesic ball
$\mathscr{B}(q,r)\subset M^{\ast}$ is given by
\begin{eqnarray*}
\mathrm{vol}(\mathscr{B}(q,r))=w_{n}\int_{0}^{r}f^{n-1}(s)ds,
\end{eqnarray*}
and moreover, the volume of the boundary $\partial\mathscr{B}(q,r)$
is given by
\begin{eqnarray*}
\mathrm{vol}(\partial\mathscr{B}(q,r))=w_{n}f^{n-1}(r),
\end{eqnarray*}
where $w_n$ denotes the $(n-1)$-volume of the unit sphere
$\mathbb{S}^{n-1}$.

We also need the following concepts, which can be found e.g. in our
previous works \cite{fmi, JM3, JM4}.

\begin{defn} \label{def-5.2}
 Given a continuous function
$\kappa:[0,l)\rightarrow \mathbb{R}$, we say that $M^{n}$ has a
radial Ricci curvature lower bound $(n-1)\kappa$ at the point $q$ if
\begin{eqnarray}  \label{def-cur1}
\mathrm{Ric}(v_x,v_x)\geq(n-1)\kappa(r(x)), \quad\forall x\in
M^{n}\backslash Cut(q)\cup \{q\} ,
\end{eqnarray}
where $\mathrm{Ric}$ is the Ricci curvature of $M^{n}$.
\end{defn}

\begin{defn} \label{def-5.3}
Given a continuous function $\kappa:[0,l)\rightarrow \mathbb{R}$, we
say that $M^n$ has a radial sectional curvature upper bound $\kappa$
at the point $q$ if
\begin{eqnarray}  \label{def-cur2}
\mathscr{K}(v_{x},V)\leq\kappa(r(x)),  \quad \forall x\in
M^{n}\setminus Cut(q)\cup \{q\} ,
\end{eqnarray}
where $V\perp{v_{x}}$, $V\in{S^{n-1}_{x}}\subseteq{T_{x}M^{n}}$, and
$\mathscr{K}(v_{x},V)$ denotes the sectional curvature of the plane
spanned by $v_{x}$ and $V$.
\end{defn}

\begin{remark}
\rm{ As pointed out in \cite[Remark 2.4]{fmi} or \cite[Remark
2.5]{JM4}, since the radial distance is $r(x)=d_{M^n}(x,q)$ for
$x=\gamma_{\xi}(s)=\exp_{q}(s\xi)$, the parameter $s$ may be seen as
the argument of the continuous function
$\kappa:[0,l)\rightarrow\mathbb{R}$ in Definition \ref{def-5.2} or
Definition \ref{def-5.3}. Therefore, one has
$\frac{d}{ds}|_{x}=\nabla r(x)=v_{x}$, which implies
(\ref{def-cur1}) and (\ref{def-cur2}) become
 \begin{eqnarray*}
\mathrm{Ric}(\frac{d}{ds},\frac{d}{ds})\geq(n-1)\kappa(s)
 \end{eqnarray*}
and
 \begin{eqnarray*}
\mathscr{K}(\frac{d}{ds},V)\leq\kappa(s),
 \end{eqnarray*}
respectively. By the way, since $x\in M^{n}\backslash Cut(q)\cup
\{q\} $, one knows that the geodesic
$\gamma_{\xi}(s)=\exp_{q}(s\xi)$, emanating from $q$ and joining
$x$, should be unit-speed and minimizing. So, if a given manifold
$M^n$ satisfies (\ref{def-cur1}) (resp., (\ref{def-cur2})), then we
say that $M^n$ has a radial Ricci curvature lower bound
$(n-1)\kappa$ (resp., a radial sectional curvature upper bound
$\kappa$) along any unit-speed minimizing geodesic starting from the
point $q\in M^n$, that is to say, for convenience and
simplification, its radial Ricci curvature is bounded from below
with respect to $q$ (resp., radial sectional curvature is bounded
from above with respect to $q$) by the continuous function
$(n-1)\kappa$ (resp., $\kappa$).
 }
\end{remark}

\begin{remark} \label{remark2-5}
\rm{ Similarly, one can define the concept that the radial
Ricci/sectional curvature has an upper/lower bound (with respect to
some point) which is given by a continuous function of the radial
distance parameter. By the way, the concept that the radial (Ricci
or sectional) curvature is bounded at some point has been improved
to a more general setting, i.e. the integral radial (Ricci or
sectional) curvature is bounded with respect to some point -- see
\cite[Definitions 2.1 and 2.2]{JM5}. In fact, the quantities defined
in \cite[Definitions 2.1 and 2.2]{JM5} can be used to measure the
deviation
 of the radial (Ricci or sectional) curvature (with
respect to some point $q\in M^n$) from a prescribed continuous
function $(n-1)\kappa(r(x))$ or  $\kappa(r(x))$, where
$r(x)=d_{M^n}(x,q)$ denotes the intrinsic distance to $q$ on the
Riemannian $n$-manifold $M^n$. }
\end{remark}

Now, we wish to point out an important truth, that is, for a
prescribed $n$-dimensional complete manifold $M^n$, one can always
construct the optimal continuous functions $\kappa_{\pm}(q,s)$ with
respect to a point $q\in M^{n}$, satisfying Definitions
\ref{def-5.2} and \ref{def-5.3} respectively. This truth has already
been shown in our previous works \cite[p. 706]{fmi}, \cite[pp.
378-379]{JM4}. However, for convenience to readers, we wish to
repeat the construction here. Recall that for $\xi\in
S_{q}^{n-1}\subseteq T_{q}M^{n}$, $\gamma_{\xi}(s)=\exp_{q}(s\xi)$
and its derivative $\gamma'_{\xi}(s)$ depend smoothly on the
variables $(s,\xi)$. Set
$\mathbb{D}_{q}:=\{(s,\xi)\in[0,\infty)\times S_{q}^{n-1}|0\leq
s<d_{\xi} \}$ with closure
$\overline{\mathbb{D}}_{q}:=\{(s,\xi)\in[0,\infty)\times
S_{q}^{n-1}|0\leq s\leq d_{\xi} \}$. Then we can define
 \begin{eqnarray} \label{slb}
 \kappa_{-}(q,s):=
\min\limits_{\{\xi|(s,\xi\in\overline{\mathbb{D}}_{q})\}}
 \frac{\mathrm{Ric}_{\gamma_{\xi}(s)}(\gamma'_{\xi}(s),\gamma'_{\xi}(s))}{n-1},\qquad
 0\leq s<\ell(q)
 \end{eqnarray}
and
 \begin{eqnarray} \label{sub}
 \kappa_{+}(q,s):=\max\limits_{\{(\xi,V)|\gamma'_{\xi}(s)\perp V,
 |V|=1\}}\mathscr{K}_{\gamma_{\xi}(s)}(\gamma'_{\xi}(s),V), \qquad 0\leq
 s<\mathrm{inj}(q).
\end{eqnarray}
If $\ell(q)<+\infty$, the above two functions can be continuously
extended to $s=\ell(q)$ and $s=\mathrm{inj}(q)$ respectively.
Furthermore, if $M^n$ is closed, then the injectivity radius
$\mathrm{inj}(M^n)=\min_{q\in M^{n}}\mathrm{inj}(q)$ is a positive
constant, and then in this case by applying the uniform continuity
of continuous functions on compact sets, one knows that
$\kappa_{\pm}(q,s)$ are continuous. Hence, for a bounded domain
$\Omega\subseteq M^{n}$, one can always find optimally continuous
bounds $\kappa_{\pm}(q,s)$ for the radial sectional and Ricci
curvatures with respect to the given point $q\in\Omega$. This
implies that the assumptions on curvatures in Definitions
\ref{def-5.2} and \ref{def-5.3} are natural and feasible. Specially,
when $n=2$, $M^2$ is a complete surface, and then
$\kappa_{\pm}(q,s)$ defined by (\ref{slb}) and (\ref{sub}) are
actually the minimum and maximum of the Gaussian curvature on
geodesic circles centered at $q$ of radius $s$ on $M^2$.

As explained in \cite[pp. 706-707]{fmi} or \cite[p. 379]{JM4}, by
using the properties of warped product manifolds, one knows that the
radial sectional curvature, and the radial component of the Ricci
tensor of the spherically symmetric manifold
$[0,l)\times_{f(s)}\mathbb{S}^{n-1}$ with the base point $q$ are
given by
\begin{eqnarray*}
\mathscr{K}\left(V,\frac{d}{ds}|_{\exp_{q}(s\xi)}\right)=R\left(\frac{d}{ds}|_{\exp_{q}(s\xi)},V,\frac{d}{ds}|_{\exp_{q}(s\xi)},V\right)
=-\frac{f''(s)}{f(s)}
\end{eqnarray*}
for $V\in T_{\xi}\mathbb{S}^{n-1}$, $|V|=1$, and
 \begin{eqnarray*}
\mathrm{Ric}\left(\frac{d}{ds}|_{\exp_{q}(s\xi)},\frac{d}{ds}|_{\exp_{q}(s\xi)}\right)=-(n-1)\frac{f''(s)}{f(s)}.
 \end{eqnarray*}
So, in this setting, Definition \ref{def-5.2} (resp., Definition
\ref{def-5.3}) is satisfied with equality in (\ref{def-cur1})
(resp., (\ref{def-cur2})) and correspondingly
$\kappa(s)=-f''(s)/f(s)$. From the above two facts, we know that in
order to define curvature tensor away from $q$, we need to require
$f\in C^{2}((0,l))$. Furthermore, if $f''(0)=0$ and $f$ is $C^3$ at
$s=0$, then one has $\lim_{s\rightarrow0}\kappa(s)=-f'''(0)$.
Although $\nabla r(x)$ is not defined at $x=q$, $\kappa(s)$ is
usually required to be continuous at $s=0$, which is equal to
require $f$ to be $C^3$ at $s=0$.

Now, we wish to recall two Bishop-type volume comparison theorems,
which correspond to Theorem 3.3, Corollary 3.4, and Theorem 4.2 in
\cite{fmi} (equivalently, Theorem 2.2.3, Corollary 2.2.4 and Theorem
2.3.2 in \cite{JM3}). Define a quantity on $M^{n}\setminus Cut(q)$
by
 \begin{eqnarray*}
\theta(s,\xi)=\left[\frac{J(s,\xi)}{f(s)}\right]^{n-1}.
\end{eqnarray*}
Let $a(\xi):=\min \{d_{\xi}, r_0\}$ on $M^{n}$, and we always assume
$r_0< \min\{ \ell(q), l\}$. One has the following facts:

\begin{theorem}
[Generalized Bishop's comparison theorem I] {\rm{(see \cite{fmi,
JM3})}} \label{BishopI} Given $\xi \in S_q^{n-1}$, and a model space
$M^-=[0,l)\times_f \mathbb{S}^{n-1}$,
 under the curvature assumption on the radial Ricci tensor,
$\mathrm{Ric}(v_x,v_x)\geq -(n-1)f''(s)/f(s)$, for
$x=\gamma_{\xi}(s)=\exp_{q}(s\xi)$
 with $s<\min\{d_{\xi}, l\}$ (resp.\ with $s<\min\{a_{\xi}, l\}$)  the function
$\theta (s,\xi)$ is nonincreasing in $s$. In particular, for all
$s<\min\{d_\xi, l\}$~ (resp.\ $s< \min \{a(\xi), l\}$) we have
$J(s,\xi)\leq f(s)$. Furthermore, this inequality is strict for all
$s\in (s_0,s_1]$, with $0\leq s_0<s_1<\min \{d_{\xi},l\}$,
 if the above  curvature
assumption holds with
 a strict inequality for $s$ in the same interval. Besides, we have
  $$\mathrm{vol}(B(q,r_0))\leq\mathrm{vol}(\mathscr{B}_n(q^{-},r_0)),$$
with equality if and only if $B(q,r_0)$ is isometric to
$\mathscr{B}_n(q^{-},r_0)$.
\end{theorem}

\begin{theorem}
[Generalized Bishop's comparison theorem II] {\rm{(see \cite{fmi,
JM3})}} \label{BishopII}
 ~Suppose $M^n$ has a radial sectional curvature upper bound given by
 ~$\kappa(s)=-\frac{f''(s)}{f(s)}$ for ~$s<\beta\leq \min\{inj_c(q), l\}$,
where $inj_c(q)=\inf_{\xi}c_{\xi}$, with $\gamma_{\xi}(c_{\xi})$ a
first conjugate point along the geodesic
$\gamma_{\xi}(s)=\exp_{q}(s\xi)$. Then on $(0,\beta)$
\begin{eqnarray*}
\left[\frac{\sqrt{|g|}}{f^{n-1}}\right]'\geq0, \quad\quad
\sqrt{|g|}(s)\geq{f^{n-1}(s)},
 \end{eqnarray*}
and  equality occurs in the first inequality at $s_{0}\in(0,\beta)$
if and only if
 \begin{eqnarray*}
\mathcal{R}=-\frac{f''(s)}{f(s)}, \quad \mathbb{A}=f(s)I
 \end{eqnarray*}
 on all of $[0,s_{0}]$.
\end{theorem}

\begin{remark}
\rm{ For any measurable subset $\Gamma$ of $S_{q}M^{n}\subset
T_{q}M^{n}$, an annulus $A^{\Gamma}_{r,R}(q)$ can be well-defined as
follows
\begin{eqnarray*}
&&A^{\Gamma}_{r,R}(q):=\{x\in M^{n}|r\leq r(x)\leq R,\\
&& \qquad \qquad \qquad
{\rm{and~any~minimal~geodesic~\gamma~from}}~q~{\rm{to}}~x~{\rm{satisfies}}~\gamma'(0)\in\Gamma\}.
\end{eqnarray*}
Clearly, if $r=0$ and $R=r_0$, then the annulus
$A^{\Gamma}_{0,r_0}(q)$ would be exactly the geodesic ball
$B(q,r_0)$ on the given Riemannian $n$-manifold $M^{n}$.
 In \cite[Theorem 3.2]{JM-ex1} (resp., \cite[Theorem 4.3]{JM-ex1}),
 we have successfully obtained a relative volume comparison for the
 annulus $A^{\Gamma}_{r,R}(q)$ on a complete Riemannian $n$-manifold
 $M^n$ having a radial Ricci curvature lower bound
 $-(n-1)f''(s)/f(s)$ with respect to the point $q\in M^n$ (resp., a radial sectional curvature upper bound
 $-f''(s)/f(s)$ with respect to the point $q$). Moreover, those two
 volume comparison results would immediately imply our Bishop-type
 volume comparisons (i.e. Theorems \ref{BishopI} and \ref{BishopII} here)
 by suitably choosing radius parameters for annuluses -- see \cite[Corollaries 3.4 and
 4.4]{JM-ex1} for details. In 2023, we attempted to improve our
 previous volume comparison results mentioned here to the setting
 that the given complete Riemannian $n$-manifold has integral radial curvature
 bounds (i.e. a more weaker curvature assumption), and this attempt
 was successful -- see \cite[Section 3]{JM5} for details.
}
\end{remark}

As mentioned in Remark \ref{remark-1.19}, the assertion in Theorem
\ref{theo-1} will be generalized to the case of nonlinear weighted
$p$-Laplacian ($1<p<\infty$) in Appendix B. For convenience to
readers, here we wish to give some fundamental facts on the
Dirichlet eigenvalue problem of the weighted $p$-Laplacian. In fact,
for a bounded domain $\Omega$ (with boundary $\partial\Omega$) on
the given Riemannian $n$-manifold $M^{n}$, one can consider the
following eigenvalue problem
\begin{eqnarray} \label{eigenp-wpl}
\left\{
\begin{array}{ll}
\Delta_{p, \phi} u + \lambda|u|^{p-2}u=0 \qquad  &\mathrm{in}~~\Omega,  \\[0.5mm]
u= 0 \qquad &\mathrm{on}~~\partial \Omega,
\end{array} \right.
\end{eqnarray}
where $\Delta_{p,
\phi}:=e^{\phi}\mathrm{div}(e^{-\phi}|\nabla\cdot|^{p-2}\nabla\cdot)$,
$1<p<\infty$, denotes the weighted $p$-Laplacian. That is to say, in
a local coordinates chart $\{x^{1},x^{2},\cdots,x^{n}\}$ of $M^{n}$,
one has
\begin{eqnarray*}
\Delta_{p,
\phi}u=\frac{e^{\phi}}{\sqrt{\det[g_{ij}]}}\sum\limits_{i,j=1}^{n}\frac{\partial}{\partial
x^{i}}\left(e^{-\phi}\sqrt{\det[g_{ij}]}g^{ij}|\nabla
u|^{p-2}\frac{\partial u}{\partial x^{j}}\right),
\end{eqnarray*}
where $|\nabla u|^{2}=\sum_{i,j=1}^{n}g^{ij}\frac{\partial}{\partial
x^{i}}\frac{\partial}{\partial x^{j}}$, and $(g^{ij})=[g_{ij}]^{-1}$
is the inverse of the metric matrix. Clearly, if $\phi=const.$, then
$\Delta_{p, \phi}$ would reduce to the classical nonlinear
$p$-Laplacian
$\Delta_{p}:=\mathrm{div}(|\nabla\cdot|^{p-2}\nabla\cdot)$, and
correspondingly, one has
\begin{eqnarray*}
\Delta_{p}u=\frac{1}{\sqrt{\det[g_{ij}]}}\sum\limits_{i,j=1}^{n}\frac{\partial}{\partial
x^{i}}\left(\sqrt{\det[g_{ij}]}g^{ij}|\nabla u|^{p-2}\frac{\partial
u}{\partial x^{j}}\right).
\end{eqnarray*}
Besides, if $p=2$, then the weighted $p$-Laplacian $\Delta_{p,
\phi}$ would become the Witten-Laplacian $\Delta_{\phi}$ exactly.
For the eigenvalue problem (\ref{eigenp-wpl}), by making suitable
adjustments to the arguments in \cite{AA, AT, BK, GAPA, LP} and
using the variational principle, it is not hard to get:
\begin{itemize}
\item (\ref{eigenp-wpl}) has a positive weak solution, which is unique modulo
the scaling, in the space $W^{1,p}_{0,\phi}(\Omega)$, the completion
of the set $C^{\infty}_{0}(\Omega)$ of smooth functions compactly
supported on $\Omega$ under the Sobolev norm
\begin{eqnarray*}
\|u\|^{\phi}_{1,p}:=\left(\int_{\Omega}|u|^{p}e^{-\phi}dv+\int_{\Omega}|\nabla
u|^{p}e^{-\phi}dv\right)^{1/p}.
\end{eqnarray*}

\item By applying the
Ljusternik-Schnirelman principle, there exists a nondecreasing
sequence of nonnegative eigenvalues
$\{\lambda_{i,p}^{\phi}(\Omega)\}$ tending to $\infty$ as
$i\rightarrow\infty$.

\item The first Dirichlet eigenvalue $\lambda_{1,p}^{\phi}(\Omega)$  is simple, isolated, and eigenfunctions
associated with $\lambda_{1,p}^{\phi}(\Omega)$ do not change sign.
Besides, $\lambda_{1,p}^{\phi}(\Omega)$ can be characterized by
\begin{eqnarray} \label{chr-2}
\lambda_{1,p}^{\phi}(\Omega)=\inf\left\{\frac{\int_{\Omega}|\nabla
u|^{p}e^{-\phi}dv}{\int_{\Omega}|
u|^{p}e^{-\phi}dv}=\frac{\int_{\Omega}|\nabla
u|^{p}d\mu}{\int_{\Omega}|u|^{p}d\mu}\Big{|}u\neq0,~u\in
W^{1,p}_{0,\phi}(\Omega)\right\}.
\end{eqnarray}

\item The set of eigenvalues is closed. The eigenvalue
$\lambda_{2,p}^{\phi}(\Omega)$ is the second eigenvalue, i.e.
$\lambda_{2,p}^{\phi}(\Omega)=\inf\{\lambda|\lambda~{\rm{is~an~eigenvalue~of~(\ref{eigenp-wpl})~and}}~\lambda>\lambda_{1,p}^{\phi}(\Omega)\}$.

\end{itemize}
In \cite[Section 3]{DMWX}, we have given several lower bounds for
the first eigenvalue of weighted $p$-Laplacian on submanifolds with
locally bounded weighted mean curvature.

\section{Weighted heat kernel comparison theorems} \label{Sect3}
\renewcommand{\thesection}{\arabic{section}}
\renewcommand{\theequation}{\thesection.\arabic{equation}}
\setcounter{equation}{0}

Given an $n$-dimensional smooth metric measure space
$(M^{n},\langle\cdot,\cdot\rangle,e^{-\phi}dv)$ and any bounded
domain $\Omega$ in this space, by using the divergence theorem, it
is easy to get the following Green's formula
\begin{eqnarray*}
\int_{\Omega}h_{1}\Delta_{\phi}h_{2}e^{-\phi}dv=-\int_{\Omega}\langle\nabla
h_{1},\nabla h_{2}\rangle
e^{-\phi}dv=\int_{\Omega}h_{2}\Delta_{\phi}h_{1}e^{-\phi}dv
\end{eqnarray*}
for $h_{1},h_{2}\in C^{\infty}_{0}(\Omega)$, i.e.
\begin{eqnarray} \label{GF}
\int_{\Omega}h_{1}\Delta_{\phi}h_{2}d\mu=-\int_{\Omega}\langle\nabla
h_{1},\nabla h_{2}\rangle
d\mu=\int_{\Omega}h_{2}\Delta_{\phi}h_{1}d\mu.
\end{eqnarray}
It is not hard to check that Green's formula (\ref{GF}) remains
valid if the compact support condition for $h_{1},h_{2}$ was
replaced by the Neumann boundary condition on the boundary
$\partial\Omega$. By (\ref{GF}), one knows that the operator
$\Delta_{\phi}$ in\footnote{~In the weighted setting, similar as the
symbol $L^{1}_{\phi}$ in Section \ref{Sect1}, $L^{2}_{\phi}$ means
doing $L^{2}$ integrals with respect to the weighted measure
$d\mu=e^{-\phi}dv$. We will not mention this explanation in the
sequel any more.} $L^{2}_{\phi}$ with the space
$C^{\infty}_{0}(\Omega)$ is symmetric and nonpositive definite. It
is natural to ask:
\begin{itemize}
\item \emph{Whether the operator $\Delta_{\phi}|_{C^{\infty}_{0}(\Omega)}$ has a self-adjoint extension in $L^{2}_{\phi}$ and whether it is essentially
self-adjoint?}
\end{itemize}
This question is crucial for the eigenvalue problem
(\ref{eigenp-w1}) and the answer is affirmative. In fact, as
mentioned in A. Grigor'yan's survey \cite[Section 2]{GA}, by
suitably making
 adjustments to those related arguments in e.g. \cite{CY,
DEB, GMP, SRS}, it is not hard to get the following result.

\begin{theorem} \label{theo-3-1} {\rm{(\cite[Section 2]{GA})}}
The operator $\Delta_{\phi}|_{W^{2,2}_{0,\phi}(\Omega)}$ is a
self-adjoint nonpositive definite operator in $L^{2}_{\phi}$, where
the space $W^{2,2}_{0,\phi}(\Omega)$ is defined as
$W^{2,2}_{0,\phi}(\Omega):=\{\rho\in
W^{1,2}_{0,\phi}(\Omega)|\Delta_{\phi}\rho\in
L^{2}_{\phi}(\Omega)\}$. Moreover, the operator
$\Delta_{\phi}|_{W^{2,2}_{0,\phi}(\Omega)}$ is a unique self-adjoint
extension of $\Delta_{\phi}|_{C^{\infty}_{0}(\Omega)}$ with the
domain in $W^{1,2}_{0,\phi}(\Omega)$. If in addition $\Omega$ is
geodesically complete, then the operator
$\Delta_{\phi}|_{C^{\infty}_{0}(\Omega)}$ is essentially
self-adjoint, i.e., $\Delta_{\phi}|_{W^{2,2}_{0,\phi}(\Omega)}$ is a
unique self-adjoint extension of
$\Delta_{\phi}|_{C^{\infty}_{0}(\Omega)}$.
\end{theorem}

\noindent Once we have Theorem \ref{theo-3-1} for the extension of
the operator $\Delta_{\phi}|_{C^{\infty}_{0}(\Omega)}$. The spectrum
of the eigenvalue problem (\ref{eigenp-w1}) is actually the spectral
structure of the self-adjoint operator
$-\Delta_{\phi}|_{W^{2,2}_{0,\phi}(\Omega)}$ in
$L^{2}_{\phi}(\Omega)$, which is only discrete and can be deduced
 from the compactness of the resolvent
$\left(-\Delta_{\phi}+I\right)^{-1}$, together with the spectral
theory of compact self-adjoint differential operators. See e.g.
\cite[Subsection 2.2]{GA} for details. Combining with the
variational principle, one can get two basic facts for the
eigenvalue problem (\ref{eigenp-w1}) listed in Section \ref{Sect1}.

As explained in \cite[Section 3]{GA}, once the self-adjoint
extension of $\Delta_{\phi}|_{C^{\infty}_{0}(\Omega)}$ has been set
up as in Theorem \ref{theo-3-1}, the heat semigroup
$e^{t\Delta_{\phi}}$ related to the Witten-Laplacian $\Delta_{\phi}$
can be well-defined. Then one can get the following properties for
the weighted heat kernel by using the tool of heat semigroup.

\begin{lemma} \label{lemma3-1} {\rm{(see e.g. \cite[Theorem 3.3 in Section 3]{GA})}}
Let $M$ be a complete Riemannian manifold and $\phi$ be a
real-valued smooth function on $M$. Denote by $\mathbb{R}^{+}$ the
set of all positive real numbers. Then there exists a weighted heat
kernel $H^{\phi}(x,y,t)\in C^{\infty}(M\times
M\times\mathbb{R}^{+})$ such
that \\
  (1) $\lim_{t\rightarrow0}H^{\phi}(x,y,t)=\delta^{\phi}_{x}(y)$,\\
  (2) $(\Delta_{\phi}-\frac{d}{dt})H^{\phi}=0$,\\
  (3) $H^{\phi}(x,y,t)=H^{\phi}(y,x,t)$,\\
  (4)
  $H^{\phi}(x,y,t)=\int_{M}H^{\phi}(x,z,t-s)H^{\phi}(z,y,s)d\mu(z)$,
  with $0<s<t$.
\end{lemma}

\begin{remark}
\rm{ It is easy to see that (1)-(2) in Lemma \ref{lemma3-1} have
close relation with the existence of the weight heat kernel. In
\cite[Theorem 3.3 in Section 3]{GA}, the property (3) is the
\emph{symmetry} of the weight heat kernel, while the property (4) is
the \emph{semigroup identity} of the weight heat kernel. In
\cite[Section 3]{GA}, A. Grigor'yan mentioned at least two
approaches to prove the existence of the weight heat kernel --- the
classical one is the usage of a parametrix of the heat equation, and
the other one is the usage of eigenfunction expansion formula in an
exhaustion of relatively compact open sets
$\{\Omega_i\}_{i=1}^{\infty}$ of $M$. For details, one can check
e.g. \cite{BGV, IC, CY, DEB, SRS} as references. }
\end{remark}

The following strong maximum (resp., minimum) principle is needed.

\begin{theorem} \label{theo-3-5}
Given a Riemannian manifold $M$ and the Witten-Laplacian
$\Delta_{\phi}$, and the associated weighted heat operator
$\mathcal{L}^{\phi}=\Delta_{\phi}-\frac{\partial}{\partial t}$,
where $\phi$ is a smooth real-valued function on $M$. Let $u(x,t)$
be a bounded continuous function on $M\times[0,T]$, which is $C^2$
with respect to the variable $x\in M$, and $C^1$ with respect to
$t\in[0,T]$, and which satisfies
 \begin{eqnarray*}
 \mathcal{L}^{\phi}u\geq0 \qquad (\mathcal{L}^{\phi}u\leq0)
 \end{eqnarray*}
on $M\times(0,T)$. If there exists $(x_0,t_0)$ in $M\times(0,T]$
such that
\begin{eqnarray*}
u(x_0,t_0)=\sup\limits_{M\times[0,T]}u(x,t) \qquad
\left(resp.,~u(x_0,t_0)=\inf\limits_{M\times[0,T]}u(x,t)\right),
\end{eqnarray*}
then $ u|_{M\times[0,t_0]}=u(x_0,t_0) $.
\end{theorem}

\begin{proof}
It is easy to check that $\mathcal{L}^{\phi}$ is uniformly
parabolic, and then the assertion in Theorem \ref{theo-3-1} directly
follows from the standard theory of second-order parabolic partial
differential equations (see e.g. \cite[p. 180]{IC} or \cite[Chapter
7]{E1}).
\end{proof}

We also need:

\begin{lemma} \label{lemma3-3}
For the model space $M^{+}$ (resp., $M^{-}$), if furthermore the
potential function $\phi$ has \textbf{Property 1}, its weighted heat
kernel $H^{\phi}_{+}(q^{+},z,t)=H^{\phi}_{+}(r_{1},t)$ (resp.,
$H^{\phi}_{-}(q^{-},z,t)=H^{\phi}_{-}(r_{2},t)$) depends only on
variables $r_{1}=d_{M^{+}}(q^{+},z)$ (resp.,
$r_{2}=d_{M^{-}}(q^{-},z)$) and $t$. Moreover, for all $t>0$, we
have
 \begin{eqnarray*}
 \frac{\partial}{\partial r_1}H^{\phi}_{+}(r_{1},t)<0 \qquad
 \left(\frac{\partial}{\partial r_2}H^{\phi}_{-}(r_{2},t)<0\right).
 \end{eqnarray*}
\end{lemma}

\begin{proof}
If the potential function $\phi$ has \textbf{Property 1}, then the
weighted heat kernel $H_{+}^{\phi}$ (resp, $H^{\phi}_{-}$) on a
geodesic ball $\mathscr{B}_{n}(q^{+},\cdot)$ (resp,
$\mathscr{B}_{n}(q^{-},\cdot)$) of the model space $M^{+}$ (resp,
$M^{-}$) depends on the space variable $r_1$ (resp., $r_2$) can be
observed by the decomposition of the Witten-Laplacian on
$\mathscr{B}_{n}(q^{+},\cdot)$ (resp,
$\mathscr{B}_{n}(q^{-},\cdot)$) in geodesic spherical coordinates at
$q^{+}$ (resp, $q^{-}$) --- see the $3^{\rm{rd}}$ line in the proof
of Lemma \ref{lemma6-1} for the detailed expression. The
monotonicity of $H_{+}^{\phi}$ (resp, $H^{\phi}_{-}$) with respect
to the time variable can be obtained by using an almost the same
argument to that in the proof of \cite[Lemma 2.3]{CY}. But there are
two differences we wish to point out as follows:
\begin{itemize}
\item In the non-weighted case, \cite{CY} has used the fact that for
sufficiently small $0<r_{1}, t<\epsilon$, the heat kernel
$H_{+}(r_{1},t)$ is asymptotical to be
 \begin{eqnarray*}
 H_{+}(r_{1},t) \sim \frac{1}{(4\pi t)^{n/2}}\times
\exp\left(\frac{-r_{1}^2}{4t}\right)\sum\limits_{i=0}^{\infty}a_{i}(r_1)t^{i},
 \end{eqnarray*}
where $a_{i}(r_1)$ are smooth functions on $M^{+}$. This directly
implies the expression (2.11) in \cite{CY}. In the weighted case,
\cite[Theorem 3.9]{GA} tells us that for sufficiently small
$0<r_{1}, t<\epsilon$, there exists a smooth positive function
$u(r_1)$ such that the weighted heat kernel $H_{+}^{\phi}(r_{1},t)$
satisfies
 \begin{eqnarray*}
 H_{+}^{\phi}(r_{1},t) \sim \frac{1}{(4\pi t)^{n/2}}\times
\exp\left(\frac{-r_{1}^2}{4t}\right)u(r_1),
 \end{eqnarray*}
 which leads to the truth that the analysis process in the proof of \cite[Lemma
 2.3]{CY} before the equation (2.20) can be
transplanted to the non-weighted case without any big obstacle. Same
story happens to $H^{\phi}_{-}(r_{2},t)$. By the way, a detailed
proof of the above asymptotic expansion for the weighted heat kernel
can be found in \cite{BGM, Gpb}.

\item For the second-order parabolic equation (2.20) in \cite{CY} for the quantity $\partial H_{+}/\partial
r_1$, the coefficient $-m(r_1)$ of its first-order derivative term
is given by
\begin{eqnarray*}
-(n-1)\frac{f'(r_1)}{f(r_1)},
\end{eqnarray*}
with $f$ appearing in the metric of the model considered therein.
Actually, if one checks the details in \cite{CY} carefully, he or
she would find that a spherically symmetric manifold must be an open
or closed Ricci model with respect to its base point (i.e. our model
space used here is more accurate than the ones used in \cite{CY}),
and then except regularity assumptions, the function $f$ in $m(r_1)$
of the equation (2.20) in \cite{CY} should have the same form as our
warping function $f$ here determined by the initial value problem
(\ref{ODE}). Now, as the equation (2.20) in \cite{CY}, if one
attempts to derive an equation for the quantity $\partial
H^{\phi}_{+}/\partial r_1$, the coefficient of its first-order
derivative term would become
\begin{eqnarray*}
-\left[(n-1)\frac{f'(r_1)}{f(r_1)}-\phi'(r_1)\right],
\end{eqnarray*}
which has no effect on the uniform parabolicity of that equation.
This implies that for the weighted case, similar to what Cheeger-Yau
have done for the equation (2.20) in \cite{CY}, we can also use the
maximum principle for the second-order uniformly parabolic partial
differential equations to deal with the equation of $\partial
H^{\phi}_{+}(r_{1},t)/\partial r_1$, and then the result
\begin{eqnarray*}
 \frac{\partial}{\partial r_1}H^{\phi}_{+}(r_{1},t)<0, \quad \forall
 t>0
 \end{eqnarray*}
 follows. Same
story happens to $H^{\phi}_{-}(r_{2},t)$.
\end{itemize}
Our proof is finished.
\end{proof}

Now, we have:

\begin{proof} [Proof of Theorem \ref{maintheorem}]
We wish to use a similar argument to that in the proof of our
previous work \cite[Theorem 6.6]{JM4}.

Assume that $M^n$ has a radial Ricci curvature lower bound
$(n-1)\kappa(s)$ w.r.t. some point $q\in M^n$,
$r_{0}<\min\{\ell(q),l\}$, and the potential function $\phi$
satisfies \textbf{Property 1}. Since the geodesic ball $B(q,r_0)$
maybe has points on the cut-locus of $q$, which leads to the
invalidity of the path of linear transformations $\mathbb{A}$
defined as Section \ref{Sect2}, we need to use a limit procedure
shown in \cite{CY} to avoid this problem. By applying Lemma
\ref{lemma3-1}, one can get
 \begin{eqnarray} \label{PT16-1}
 &&H^{\phi}(q,y,t)-
H^{\phi}_{-}\left(d_{M^{-}}\left(q^{-},z\right),t\right)
=H^{\phi}(q,y,t)-
H^{\phi}_{-}\left(d_{M^{n}}\left(q,y\right),t\right)\nonumber\\
  &=& H^{\phi}(q,y,t)-
H^{\phi}_{-}\left(r_{2}\left(q,y\right),t\right)\nonumber\\
&=&
-\int\limits_{0}^{t}\int\limits_{B(q,r_0)}\frac{\partial}{\partial
s}\left[H^{\phi}_{-}\left(r_{2}\left(q,x\right),t-s\right)\right]H^{\phi}(x,y,s)d\mu(x)ds \nonumber\\
&&  +
\int\limits_{0}^{t}\int\limits_{B(q,r_0)}H^{\phi}_{-}\left(r_{2}\left(q,x\right),t-s\right)\frac{\partial
H^{\phi}}{\partial s}(x,y,s)d\mu(x)ds \nonumber\\
&=&
-\int\limits_{0}^{t}\int\limits_{B(q,r_0)}\frac{\partial}{\partial
s}\left[H^{\phi}_{-}\left(r_{2}\left(q,x\right),t-s\right)\right]H^{\phi}(x,y,s)e^{-\phi}dv(x)ds \nonumber\\
&&  +
\int\limits_{0}^{t}\int\limits_{B(q,r_0)}H^{\phi}_{-}\left(r_{2}\left(q,x\right),t-s\right)\frac{\partial
H^{\phi}}{\partial s}(x,y,s)e^{-\phi}dv(x)ds.
 \end{eqnarray}
For any $\xi\in S^{n-1}_{q}\subseteq T_{q}M^n$, let $a(\xi)=\min
\{d_{\xi}, r_0\}$ be defined as in Section \ref{Sect2}. It is not
hard to see from the definition of $d_{\xi}$ in (\ref{add-extra1})
that $a(\xi)$ is a continuous function on $S^{n-1}_{q}$. As in
\cite{CG}, we can choose a sequence of smooth functions
$a_{\epsilon}$ on  $S^{n-1}_{q}$, with $a_{\epsilon}(\xi)<a(\xi)$
for any $\xi\in S^{n-1}_{q}$, such that $a_{\epsilon}(\cdot)$
converges uniformly to $a(\cdot)$ as $\epsilon\rightarrow0$ and
moreover the set
\begin{eqnarray*}
V_{\epsilon}=\{\exp_{q}(t\xi)|t\leq a_{\epsilon}(\xi)\}
\end{eqnarray*}
is compact. Clearly, $V_{\epsilon}$ is within the cut locus of $q$.
Then it follows that (\ref{PT16-1}) becomes
\begin{eqnarray} \label{PT16-2}
&&H^{\phi}(q,y,t)-
H^{\phi}_{-}\left(d_{M^{-}}\left(q^{-},z\right),t\right)  \nonumber\\
&=& H^{\phi}(q,y,t)-
H^{\phi}_{-}\left(r_{2}\left(q,y\right),t\right)\nonumber\\
&=&\lim\limits_{\epsilon\rightarrow0}\Bigg{\{}
-\int\limits_{0}^{t}\int\limits_{V_{\epsilon}}\frac{\partial}{\partial
s}\left[H^{\phi}_{-}\left(r_{2}\left(q,x\right),t-s\right)\right]H^{\phi}(x,y,s)d\mu(x)ds \nonumber\\
&&  +
\int\limits_{0}^{t}\int\limits_{V_{\epsilon}}H^{\phi}_{-}\left(r_{2}\left(q,x\right),t-s\right)\frac{\partial
H^{\phi}}{\partial s}(x,y,s)d\mu(x)ds\Bigg{\}} \nonumber\\
&=&\lim\limits_{\epsilon\rightarrow0}\Bigg{\{}
-\int\limits_{0}^{t}\int\limits_{V_{\epsilon}}(\Delta_{\phi})_{M^{-}}\left[H^{\phi}_{-}\left(r_{2}\left(q,x\right),t-s\right)\right]H^{\phi}(x,y,s)d\mu(x)ds \nonumber\\
&&  +
\int\limits_{0}^{t}\int\limits_{V_{\epsilon}}H^{\phi}_{-}\left(r_{2}\left(q,x\right),t-s\right)(\Delta_{\phi})_{M^{n}}H^{\phi}(x,y,s)d\mu(x)ds\Bigg{\}},
\end{eqnarray}
where $(\Delta_{\phi})_{M^{-}}$, $(\Delta_{\phi})_{M^{n}}$ are the
Witten-Laplace operators on $M^{-}$ and $M^n$, respectively.
Applying Green's formula (\ref{GF}), together with either Dirichlet
or Neumann boundary condition, it can be deduced from (\ref{PT16-2})
that
\begin{eqnarray} \label{PT16-3}
&&H^{\phi}(q,y,t)-
H^{\phi}_{-}\left(d_{M^{-}}\left(q^{-},z\right),t\right)  \nonumber\\
&=&\lim\limits_{\epsilon\rightarrow0}\Bigg{\{}\int\limits_{0}^{t}\int\limits_{V_{\epsilon}}\left[(\Delta_{\phi})_{M^{n}}H^{\phi}_{-}\left(r_{2}\left(q,x\right),t-s\right)-
(\Delta_{\phi})_{M^{-}}H^{\phi}_{-}\left(r_{2}\left(q,x\right),t-s\right)\right]\times
\nonumber\\
&&\qquad \qquad \qquad\qquad H^{\phi}(x,y,s)d\mu(x)ds\Bigg{\}}.
\end{eqnarray}
On the other hand, in the geodesic spherical coordinates near $q$ or
$q^{-}$, for function of
$r_{2}(q,y)=d_{M^{-}}(q^{-},z)=d_{M^n}(q,y)$, one has
\begin{eqnarray*}
&&(\Delta_{\phi})_{M^{n}}=\frac{\partial^2}{\partial
r_{2}^{2}}+\left[\frac{\left(\sqrt{|g|}\right)'}{\sqrt{|g|}}-\phi'(r_2)\right]\frac{\partial}{\partial
r_2}=\frac{\partial^2}{\partial
r_{2}^{2}}+\left[\frac{\left(J^{n-1}(r_{2},\xi)\right)'}{J^{n-1}(r_{2},\xi)}-\phi'(r_2)\right]\frac{\partial}{\partial
r_2},\\
&&(\Delta_{\phi})_{M^{-}}=\frac{\partial^2}{\partial
r_{2}^{2}}+\left[\frac{\left(f^{n-1}(r_2)\right)'}{f^{n-1}(r_2)}-\phi'(r_2)\right]\frac{\partial}{\partial
r_2}.
\end{eqnarray*}
Substituting the above two equalities into (\ref{PT16-3}) and
applying Lemma \ref{lemma3-3}, we can obtain
\begin{eqnarray} \label{PT16-4}
&&\qquad H^{\phi}(q,y,t)-
H^{\phi}_{-}\left(d_{M^{-}}\left(q^{-},z\right),t\right)\nonumber\\
&&=\lim\limits_{\epsilon\rightarrow0}\left\{
\int\limits_{0}^{t}\int\limits_{V_{\epsilon}}(n-1)\left[\frac{J'(r_2,\xi)}{J(r_2,\xi)}-\frac{f'(r_2)}{f(r_2)}\right]
\frac{\partial
H^{\phi}_{-}\left(r_{2}\left(q,x\right),t-s\right)}{\partial
r_2}H^{\phi}(x,y,s)d\mu(x)ds\right\} \nonumber\\
 &&\geq0,
\end{eqnarray}
where the fact
$d\theta(s,\xi)/ds=(n-1)\left[J(s,\xi)/f(s)\right]^{n-2}\left[J(s,\xi)/f(s)\right]'\leq0$
shown in Theorem \ref{BishopI} has also been used. This directly
implies $H^{\phi}(q,y,t)\geq
H^{\phi}_{-}\left(d_{M^{-}}\left(q^{-},z\right),t\right)$, which is
exactly the weighted heat kernel comparison (\ref{6-1-1-ex}). When
the equality in (\ref{6-1-1-ex}) holds at some $(y_{0},t_{0})\in
B(q,r_0)\times(0,\infty)$, by Theorem \ref{theo-3-5} one has
$H^{\phi}(q,y,t)=H^{\phi}_{-}\left(d_{M^{-}}\left(q^{-},z\right),t\right)=H^{\phi}(q,y_0,t_0)$
on $B(q,r_0)\times[0,t_0]$. Together with (\ref{PT16-4}), we know
that
 \begin{eqnarray*}
\frac{J'(r_2,\xi)}{J(r_2,\xi)}-\frac{f'(r_2)}{f(r_2)}=0
 \end{eqnarray*}
 holds on $B(q,r_0)$, which implies $J(r_2,\xi)=f(r_2)$  holds on
 $B(q,r_0)$, and moreover
 $$\mathrm{vol}(B(q,r_0))=\mathrm{vol}(\mathscr{B}_n(q^{-},r_0)).$$
 Then it follows that $B(q,r_0)$ is isometric to
 $\mathscr{B}_n(q^{-},r_0)$ by directly applying Theorem
 \ref{BishopI}. This completes the first assertion (1) of Theorem
 \ref{maintheorem}.

Assume that $M^n$ has a radial sectional curvature upper bound
$\kappa(s)$ w.r.t. $q\in M^n$, $r_{0}<\min\{\mathrm{inj}(q),l\}$,
and the potential function $\phi$ satisfies \textbf{Property 1}.
Clearly, in this case, the geodesic ball $B(q,r_0)$ must be within
the cut-locus of $q$. By applying Lemma \ref{lemma3-1}, Green's
formula (\ref{GF}), Lemma \ref{lemma3-3}, and using similar
calculations to the previous case, we have
 \begin{eqnarray} \label{PT16-5}
 &&H^{\phi}(q,y,t)-
H^{\phi}_{-}\left(d_{M^{+}}\left(q^{+},z\right),t\right)  \nonumber\\
 &=&H^{\phi}(q,y,t)-
H^{\phi}_{+}\left(d_{M^{n}}\left(q,y\right),t\right)\nonumber\\
  &=& H^{\phi}(q,y,t)-
H^{\phi}_{+}\left(r_{1}\left(q,y\right),t\right)\nonumber\\
&=&
-\int\limits_{0}^{t}\int\limits_{B(q,r_0)}\frac{\partial}{\partial
s}\left[H^{\phi}_{+}\left(r_{1}\left(q,x\right),t-s\right)\right]H^{\phi}(x,y,s)d\mu(x)ds \nonumber\\
&&  +
\int\limits_{0}^{t}\int\limits_{B(q,r_0)}H^{\phi}_{+}\left(r_{1}\left(q,x\right),t-s\right)\frac{\partial
H^{\phi}}{\partial s}(x,y,s)d\mu(x)ds \nonumber\\
&=&\int\limits_{0}^{t}\int\limits_{B(q,r_0)}\left[(\Delta_{\phi})_{M^{n}}H^{\phi}_{+}\left(r_{1}\left(q,x\right),t-s\right)-
(\Delta_{\phi})_{M^{+}}H^{\phi}_{+}\left(r_{1}\left(q,x\right),t-s\right)\right]\times\nonumber\\
&&\qquad \qquad \qquad\qquad H^{\phi}(x,y,s)d\mu(x)ds \nonumber\\
&=&\int\limits_{0}^{t}\int\limits_{B(q,r_0)}\left[\frac{\left(\sqrt{|g|}\right)'}{\sqrt{|g|}}-\frac{\left(f^{n-1}(r_1)\right)'}{f^{n-1}(r_1)}\right]
\frac{\partial
H^{\phi}_{+}\left(r_{1}\left(q,x\right),t-s\right)}{\partial
r_1}H^{\phi}(x,y,s)d\mu(x)ds\nonumber\\
 &\leq&0,
 \end{eqnarray}
where the fact
 $$\left[\frac{\sqrt{|g|}}{f^{n-1}}\right]'\geq0$$
 shown in Theorem \ref{BishopII} has also been used. This
implies $H^{\phi}(q,y,t)\leq
H^{\phi}_{+}\left(d_{M^{+}}\left(q^{+},z\right),t\right)$, which is
exactly the weighted heat kernel comparison (\ref{6-1-2-ex}). When
the equality in (\ref{6-1-2-ex}) holds at some $(y_{0},t_{0})\in
B(q,r_0)\times(0,\infty)$, by Theorem \ref{theo-3-5} one has
$H^{\phi}(q,y,t)=H^{\phi}_{+}\left(d_{M^{+}}\left(q^{+},z\right),t\right)=H^{\phi}(q,y_0,t_0)$
on $B(q,r_0)\times[0,t_0]$. Together with (\ref{PT16-5}), we know
that
 \begin{eqnarray*}
\frac{\left(\sqrt{|g|}\right)'}{\sqrt{|g|}}=\frac{\left(f^{n-1}(r_1)\right)'}{f^{n-1}(r_1)}
 \end{eqnarray*}
holds on $B(q,r_0)$, which implies $\sqrt{|g|}(r_1)=f^{n-1}(r_1)$
holds on $B(q,r_0)$. By directly applying Theorem \ref{BishopII},
one knows that $\mathbb{A}(r_{1},\xi)=f(r_{1})I$ holds on
$B(q,r_0)$, which immediately implies that
$\mathrm{vol}(B(q,r_0))=\mathrm{vol}(\mathscr{B}_n(q^{+},r_0))$, and
$B(q,r_0)$ is isometric to
 $\mathscr{B}_n(q^{+},r_0)$. This completes the second assertion (2) of Theorem
 \ref{maintheorem}.
\end{proof}

\section{Applications in spectral geometry} \label{Sect4}
\renewcommand{\thesection}{\arabic{section}}
\renewcommand{\theequation}{\thesection.\arabic{equation}}
\setcounter{equation}{0}

Theorem \ref{theo-3-1} tells us that if $\Omega\subseteq M^n$ is
geodesically complete, with $M^n$ a given complete Riemannian
$n$-manifold, then $\Delta_{\phi}|_{W^{2,2}_{0,\phi}(\Omega)}$ is a
unique self-adjoint extension of
$\Delta_{\phi}|_{C^{\infty}_{0}(\Omega)}$, and this fact can be used
to prove that the eigenvalue problem (\ref{eigenp-w1}) only has
discrete spectrum, which consists of eigenvalues given in the
sequence (\ref{sequen-1}). Let $\psi_{i}$ be an eigenfunction
corresponding to the $i$-th Dirichlet eigenvalue
$\lambda_{i,\phi}(\Omega)$ such that $\{\psi_{i}\}_{i=1}^{\infty}$
forms an orthonormal basis in $L^{2}_{\phi}(\Omega)$. From
\cite[Section 3]{GA}, one knows that the weighted heat kernel
$H^{\phi}(x,y,t)$ on $\Omega$ has the following eigenfunction
expansion formula
 \begin{eqnarray} \label{4-1}
H^{\phi}(x,y,t)=\sum\limits_{i=1}^{\infty}e^{-\lambda_{i,\phi}(\Omega)t}\psi_{i}(x)\psi_{i}(y).
 \end{eqnarray}
In fact, the formula (\ref{4-1}) was derived in \cite[Section 3]{GA}
by using several properties of the heat semigroup generated by the
Witten-Laplacian $\Delta_{\phi}$. One can also check our recent work
\cite[Section 5]{ZLM} for a simpler way to get the formula
(\ref{4-1}) and also several other properties of the weighted heat
kernel.

Now, by mainly using the formula (\ref{4-1}), we have:

\begin{proof} [A proof of Cheng-type eigenvalue comparisons (\ref{ECT-1}) and
(\ref{ECT-2})] Assume that $M^n$ has a radial Ricci curvature lower
bound $(n-1)\kappa(s)$ w.r.t. some point $q\in M^n$,
$r_{0}<\min\{\ell(q),l\}$, and the potential function $\phi$
satisfies \textbf{Property 1}. Then by Theorem \ref{maintheorem}, we
have
\begin{eqnarray} \label{4-2}
H^{\phi}(q,q,t)\geq H^{\phi}_{-}(r_{2}(q,q),t)=H^{\phi}_{-}(0,t)
\end{eqnarray}
for all $t>0$.  By applying the expansion formula (\ref{4-1}), one
has
 \begin{eqnarray*}
H^{\phi}(q,q,t)=\sum\limits_{i=1}^{\infty}e^{-\lambda_{i,\phi}t}\psi_{i}^{2}(q),\\
H^{\phi}_{-}(0,t)=\sum\limits_{i=1}^{\infty}e^{-\widetilde{\lambda}_{i,\phi}t}\widetilde{\psi}_{i}^{2}(0),
 \end{eqnarray*}
where $\lambda_{i,\phi}=\lambda_{i,\phi}(B(q,r_0))$,
$\widetilde{\lambda}_{i,\phi}=\lambda_{i,\phi}\left(\mathscr{B}_{n}(q^{-},r_0)\right)$,
and $\psi_{i}$, $\widetilde{\psi}_{i}$ are eigenfunctions
corresponding to $\lambda_{i,\phi}(B(q,r_0))$,
$\lambda_{i,\phi}\left(\mathscr{B}_{n}(q^{-},r_0)\right)$
respectively. Putting the above two equalities into (\ref{4-2})
results in
\begin{eqnarray*}
&&e^{-\lambda_{1,\phi}t}\left[\psi_{1}^{2}(q)+e^{-(\lambda_{2,\phi}-\lambda_{1,\phi})t}\psi_{2}^{2}(q)+
e^{-(\lambda_{3,\phi}-\lambda_{1,\phi})t}\psi_{3}^{2}(q)+\cdots\right]\\
&& \qquad\qquad \geq
e^{-\widetilde{\lambda}_{1,\phi}t}\left[\widetilde{\psi}_{1}^{2}(0)+e^{-(\widetilde{\lambda}_{2,\phi}-\widetilde{\lambda}_{1,\phi})t}\widetilde{\psi}_{2}^{2}(0)+
e^{-(\widetilde{\lambda}_{3,\phi}-\widetilde{\lambda}_{1,\phi})t}\widetilde{\psi}_{3}^{2}(0)+\cdots\right],
\end{eqnarray*}
which implies
\begin{eqnarray} \label{4-3}
\psi_{1}^{2}(q)+e^{-(\lambda_{2,\phi}-\lambda_{1,\phi})t}\psi_{2}^{2}(q)+\cdots\geq
e^{(\lambda_{1,\phi}-\widetilde{\lambda}_{1,\phi})t}\left[\widetilde{\psi}_{1}^{2}(0)+e^{-(\widetilde{\lambda}_{2,\phi}-\widetilde{\lambda}_{1,\phi})t}\widetilde{\psi}_{2}^{2}(0)+
\cdots\right].
\end{eqnarray}
Since $\psi_{1}^{2}(q)>0$, $\widetilde{\psi}_{1}^{2}(0)$, and
$\lambda_{i,\phi}>\lambda_{1,\phi}$,
$\widetilde{\lambda}_{i,\phi}>\widetilde{\lambda}_{1,\phi}$ for any
$i\geq2$, letting $t\rightarrow\infty$ in (\ref{4-3}) leads to
\begin{eqnarray*}
\lambda_{1,\phi}-\widetilde{\lambda}_{1,\phi}\leq0,
\end{eqnarray*}
 which is equivalent to
 \begin{eqnarray*}
\lambda_{1,\phi}(B(q,r_0))\leq\lambda_{1,\phi}\left(\mathscr{B}_{n}(q^{-},r_0)\right).
 \end{eqnarray*}

 On the other hand, by applying Theorem \ref{maintheorem} and using a
 similar argument as above, one can obtain that for
$r_{0}<\min\{\mathrm{inj}(q),l\}$, the inequality
\begin{eqnarray*}
\lambda_{1,\phi}(B(q,r_0))\geq\lambda_{1,\phi}\left(\mathscr{B}_{n}(q^{+},r_0)\right).
 \end{eqnarray*}
holds when $M^n$ has a radial sectional curvature upper bound
$\kappa(s)$ w.r.t. $q\in M^n$, and the potential function $\phi$
satisfies \textbf{Property 1}. This completes the proof.
\end{proof}

\begin{remark}
\rm{ Similar to what we have pointed out in \cite[Remark 6.9]{JM4},
one cannot get the rigidity characterization for the equality in
Cheng-type eigenvalue comparisons (\ref{ECT-1}) or (\ref{ECT-2})
through the approach used in this section. In fact, if
$\lambda_{1,\phi}(B(q,r_0))=\lambda_{1,\phi}\left(\mathscr{B}_{n}(q^{-},r_0)\right)$
here, we can only get
$\lim_{t\rightarrow\infty}H^{\phi}(q,q,t)=\lim_{t\rightarrow\infty}H^{\phi}_{-}(0,t)$.
We are not sure whether there exists some $t_0$ such that
$H^{\phi}(q,q,t_0)=H^{\phi}_{-}(0,t_0)$ or not, which leads to the
fact that we cannot use the rigidity characterization for the
weighted heat kernel comparison (\ref{6-1-1-ex}) in Theorem
\ref{maintheorem}. Same story happens to
$\lambda_{1,\phi}(B(q,r_0))=\lambda_{1,\phi}\left(\mathscr{B}_{n}(q^{+},r_0)\right)$.
By the way, the isometric rigidity characterizations for the
eigenvalue comparisons (\ref{ECT-1}) and (\ref{ECT-2}) in Theorems
\ref{theo-1}-\ref{theo-2} will be proven in Appendix B using another
approach. }
\end{remark}

Except eigenvalue comparisons, the weighted heat kernel also has
some other applications in the field of spectral geometry --- such
as eigenvalue estimates, spectral isoperimetric inequalities, and so
on. We wish to refer readers to the survey \cite{GA} for this
viewpoint.

\section{Appendix A}
\renewcommand{\thesection}{\arabic{section}}
\renewcommand{\theequation}{\thesection.\arabic{equation}}
\setcounter{equation}{0}

In this appendix, we derive the explicit expression for the weighted
heat kernel in Example \ref{example4}, and the method is standard
separation of variables. Assume that the weighted heat kernel is of
the form
\begin{eqnarray*}
H^{\phi}(x_{0},y,t)=H^{\phi}(|y-x_{0}|,t)&=&\varphi(|y-x_{0}|)\psi(t)\times\exp\left(-\frac{|y-x_{0}|^2}{4t}\right)\\
&=&H^{\phi}(r,t)\\
&=&\varphi(r)\psi(t)\times\exp\left(-\frac{r^2}{4t}\right).
\end{eqnarray*}
By a direct calculation, one has
\begin{eqnarray*}
&&H^{\phi}_{t}=\varphi
e^{-\frac{r^2}{4t}}\left(\psi_{t}+\psi\frac{r^2}{4t}\right),\\
&&H^{\phi}_{r}=\psi e^{-\frac{r^2}{4t}}\left(\varphi_{r}-\varphi\frac{r}{2t}\right),\\
&&H^{\phi}_{rr}=\psi
e^{-\frac{r^2}{4t}}\left[\varphi_{rr}-\varphi_{r}\frac{r}{t}+\varphi\left(\frac{r^2}{4t^{2}}-\frac{1}{2t}\right)\right].
\end{eqnarray*}
So, the weighted heat equation
$H^{\phi}_{t}=\Delta_{\phi}H^{\phi}=H^{\phi}_{rr}-\phi_{r}H^{\phi}_{r}$,
with $\phi(r)=kr$, $k=\pm1$, implies
\begin{eqnarray*}
\varphi\left(\psi_{t}+\psi\frac{r^2}{4t}\right)=\psi\left[\varphi_{rr}-\varphi_{r}\frac{r}{t}+\varphi\left(\frac{r^2}{4t^{2}}-\frac{1}{2t}\right)\right]
-k\psi\left(\varphi_{r}-\varphi\frac{r}{2t}\right).
\end{eqnarray*}
It then follows that
\begin{eqnarray*}
\frac{\psi_{t}}{\psi}=\frac{\varphi_{rr}-k\varphi_{r}}{\varphi}-\frac{r}{2t}\cdot\frac{2\varphi_{r}-k\varphi}{\varphi}-\frac{1}{2t}.
\end{eqnarray*}
Hence, there exist constants $C_1$, $C_2$ such that
 \begin{eqnarray*}
\frac{\varphi_{rr}-k\varphi_{r}}{\varphi}=C_{1}, \qquad
\frac{(2\varphi_{r}-k\varphi)r}{\varphi}=C_{2}, \qquad
\frac{\psi_{t}}{\psi}=C_{1}-\frac{C_{2}+1}{2t}.
 \end{eqnarray*}
Solving the above ODEs yields
\begin{eqnarray*}
&&\varphi(r)=C_{3}\cdot\exp\left(\frac{1}{2}kr\right),\\
&&\psi(t)=C_{4}\cdot\frac{1}{\sqrt{t}}\exp\left(-\frac{4}{t}\right),
\end{eqnarray*}
where $C_{3}$ is independent of $r$, and $C_{4}$ is independent of
$t$. Using the initial condition
\begin{eqnarray*}
\lim\limits_{t\rightarrow0}H^{\phi}(x_{0},y,t)=\delta^{\phi}_{y}(x_0)
\end{eqnarray*}
at the point $x_{0}$, one can get
\begin{eqnarray*}
C_{3}C_{4}=\frac{1}{2\sqrt{\pi}}e^{\frac{1}{2}kr},
\end{eqnarray*}
which implies that $C_4$ should depend on $r$.

Summing up, the weighted heat kernel in Example \ref{example4} is
explicitly given by
\begin{eqnarray*}
H^{\phi}(r,t)=\frac{e^{\pm r}\cdot e^{-t/4}}{(4\pi
t)^{1/2}}\times\exp\left(-\frac{r^2}{4t}\right),
\end{eqnarray*}
which is equivalent to say
\begin{eqnarray*}
H^{\phi}(x_{0},y,t)=H^{\phi}(|y-x_{0}|,t)=\frac{e^{\pm|y-x_{0}|}\cdot
e^{-t/4}}{(4\pi
t)^{1/2}}\times\exp\left(-\frac{|y-x_{0}|^2}{4t}\right).
\end{eqnarray*}

\section{Appendix B}
\renewcommand{\thesection}{\arabic{section}}
\renewcommand{\theequation}{\thesection.\arabic{equation}}
\setcounter{equation}{0}

In order to prove Theorem \ref{theo-1}, we need the following
property.

\begin{lemma} \label{lemma6-1}
Assume that the potential function $\phi$ has \textbf{Property 1}.
The eigenfunction corresponding to the first Dirichlet eigenvalue
 of the
Witten-Laplacian on the geodesic ball $\mathscr{B}_{n}(q^{-},r_0)$
(resp., $\mathscr{B}_{n}(q^{+},r_0)$) may be chosen to be
nonnegative and is a radial function $\psi(s)$ satisfying
$\psi'(s)<0$ for $0<s\leq r_{0} <\min\{\ell(q),l\}$, with $\ell(q)$
defined as in (\ref{key-def1}) (resp., for $0<s\leq r_{0}
<\min\{\mathrm{inj}(q),l\}$, with $\mathrm{inj}(q)$ defined by
(\ref{inj-R})).
\end{lemma}

\begin{proof}
The Witten-Laplacian on $\mathscr{B}_{n}(q^{-},r_0)$ in geodesic
spherical coordinates at $q^{-}$ is given by
\begin{eqnarray*}
\Delta_{\phi}=\frac{d^{2}}{ds^2}+(n-1)\frac{f'(s)}{f(s)}\frac{d}{ds}+\frac{1}{f^{2}(s)}\Delta_{\mathbb{S}^{n-1}}-\langle\phi'(s),\nabla\cdot\rangle,
\end{eqnarray*}
where $\Delta_{\mathbb{S}^{n-1}}$ denotes the Laplacian on
$\mathbb{S}^{n-1}$ with respect to the round metric. Using a similar
argument to Courant's nodal domain theorem for the Laplacian shown
in \cite[Chapter I]{IC}, one knows that the first Dirichlet
eigenvalue $\lambda_{1,\phi}\left(\mathscr{B}_{n}(q^{-},r_0)\right)$
of the Witten-Laplacian is simple (i.e. the eigenspace of this
eigenvalue has dimension $1$) and its eigenfunction $\psi$ would not
change sign on $\mathscr{B}_{n}(q^{-},r_0)$. One can also check the
proof of our recent work \cite[Theorem 1.2]{CMW} for details if he
or she is not familiar with this content. Without loss of
generality, one may choose that $\psi$ is nonnegative on
$\mathscr{B}_{n}(q^{-},r_0)$. More precisely, $\psi$ is strictly
positive within $\mathscr{B}_{n}(q^{-},r_0)$ and vanishes on the
boundary $\partial\mathscr{B}_{n}(q^{-},r_0)$. Since the dimension
of the eigenspace corresponding to
$\lambda_{1,\phi}\left(\mathscr{B}_{n}(q^{-},r_0)\right)$ is $1$, it
follows that $\psi$ is radial and satisfies
\begin{eqnarray*}
\frac{d^{2}\psi}{ds^2}+\left[(n-1)\frac{f'(s)}{f(s)}-\phi'(s)\right]\frac{d\psi}{ds}+\lambda_{1,\phi}\left(\mathscr{B}_{n}(q^{-},r_0)\right)\psi(s)=0,
\end{eqnarray*}
i.e.
\begin{eqnarray}  \label{6-1}
\psi''(s)+\left[(n-1)\frac{f'(s)}{f(s)}-\phi'(s)\right]\psi'(s)+\lambda_{1,\phi}\psi(s)=0.
\end{eqnarray}
Set
$\alpha(s):=e^{\int_{\epsilon}^{s}\left[(n-1)\frac{f'(t)}{f(t)}-\phi'(t)\right]dt}$,
$\epsilon\rightarrow0^{+}$, and then it is not hard to know that
(\ref{6-1}) becomes
\begin{eqnarray} \label{6-2}
(\alpha\psi')'+\lambda_{1,\phi}\left(\mathscr{B}_{n}(q^{-},r_0)\right)\alpha\psi=0.
\end{eqnarray}
By a direct calculation, one has from (\ref{6-2}) that
 \begin{eqnarray*}
(\alpha\psi'')'+\left[\lambda_{1,\phi}\left(\mathscr{B}_{n}(q^{-},r_0)\right)+\left(\frac{\alpha'}{\alpha}\right)'\right]\alpha\psi'=0.
 \end{eqnarray*}
Integrating both sides of (\ref{6-2}) on the interval $(0,s)$ with
$s<r_0$, one has
\begin{eqnarray*}
\alpha(s)\psi'(s)=-\lambda_{1,\phi}\int_{0}^{s}\alpha(t)\psi(t)dt,
\end{eqnarray*}
and this implies that $\psi'(s)|_{(0,r_0]}<0$ whenever
$\psi(s)|_{(0,r_0)}>0$ and
$\lambda_{1,\phi}\left(\mathscr{B}_{n}(q^{-},r_0)\right)>0$. The
situation for the eigenfunction on the geodesic ball
$\mathscr{B}_{n}(q^{+},r_0)$ has no essential difference. Our proof
is finished.
\end{proof}

\begin{remark}
\rm{ As maybe indicated before, the reason that we require $0<s\leq
r_{0} <\min\{\ell(q),l\}$ (resp., $0<s\leq r_{0}
<\min\{\mathrm{inj}(q),l\}$) is that when the equality holds in the
Cheng-type eigenvalue comparison (\ref{ECT-1}) (resp.,
(\ref{ECT-2})), the isometric rigidity between geodesic balls
$B(q,r_0)$ and $\mathscr{B}_{n}(q^{-},r_0)$ (resp.,
$\mathscr{B}_{n}(q^{+},r_0)$) can be assured. }
\end{remark}

Now, together with Lemma \ref{lemma6-1}, we wish to prove Theorem
\ref{theo-1} by using a similar argument to that in the proof of
\cite[Theorem 3.6]{fmi}.

\begin{proof}[Proof of Theorem \ref{theo-1}]
Let $\psi$ be an eigenfunction of the Dirichlet eigenvalue
$\lambda_{1,\phi}\left(\mathscr{B}_{n}(q^{-},r_0)\right)$, which, by
Lemma \ref{lemma6-1}, is simple, radial and can be chosen to be
nonnegative. Since $\psi\circ r$ vanishes on the boundary $\partial
B(q,r_0)$, where as in Section \ref{Sect2}, $r(x)=d_{M^{n}}(q,x)$
measures the radial distance from $q$ to any point $x\in \partial
B(q,r_0)$, it is not hard to know that $\psi\circ r\in
W^{1,2}_{0,\phi}(B(q,r_0))$ and
 \begin{eqnarray*}
 \lambda_{1,\phi}(B(q,r_0))\leq\frac{\int\limits_{B(q,r_0)}\langle d\psi\circ r,d\psi\circ r\rangle e^{-\phi }dv}
 {\int\limits_{B(q,r_0)}\langle \psi\circ r,\psi\circ r\rangle e^{-\phi }dv}
 \end{eqnarray*}
 by directly using the Rayleigh
characterization (\ref{chr-1}).

Using spherical geodesic coordinates centered at $q$ under the
integrals, one can obtain
 \begin{eqnarray*}
\int\limits_{B(q,r_0)}\langle d\psi\circ r,d\psi\circ r\rangle
e^{-\phi }dv=\int\limits_{\xi\in
S_{q}^{n-1}}\left[\int\limits_{0}^{a(\xi)}\left(\frac{d\psi}{ds}\right)^{2}f^{n-1}(s)\theta(s,\xi)e^{-\phi}ds\right]d\sigma
 \end{eqnarray*}
 and
  \begin{eqnarray*}
\int\limits_{B(q,r_0)}\langle \psi\circ r,\psi\circ r\rangle
e^{-\phi }dv=\int\limits_{\xi\in
S_{q}^{n-1}}\left[\int\limits_{0}^{a(\xi)}\psi^{2}(s)f^{n-1}(s)\theta(s,\xi)e^{-\phi}ds\right]d\sigma,
  \end{eqnarray*}
 where as in Section \ref{Sect2}, $d\sigma$ is the canonical measure of $\mathbb{S}^{n-1}\equiv S_{q}^{n-1}$, $a(\xi)=\min \{d_{\xi}, r_0\}$,
 and $\theta(s,\xi)=\left[J(s,\xi)/f(s)\right]^{n-1}$ on $M^{n}\setminus
 Cut(q)$ with $f(s)$, $J(s, \xi)$ determined by (\ref{ODE}) and (\ref{J}) respectively.

 On the other hand, by a direct calculation we have
\begin{eqnarray} \label{PT116-1}
&&\int\limits_{0}^{a(\xi)}\left(\frac{d\psi}{ds}\right)^{2}f^{n-1}(s)\theta(s,\xi)e^{-\phi}ds=
\psi(s)e^{-\phi}\left(\frac{d\psi(s)}{ds}\right)f^{n-1}(s)\theta(s,\xi)\Big{|}_{0}^{a(\xi)}-
\nonumber\\
&& \qquad
\int\limits_{0}^{a(\xi)}\frac{\psi(s)}{f^{n-1}(s)\theta(s,\xi)e^{-\phi}}\frac{d}{ds}\left[f^{n-1}(s)\theta(s,\xi)e^{-\phi}\frac{d\psi(s)}{ds}\right]
f^{n-1}(s)\theta(s,\xi)e^{-\phi}ds \qquad\qquad
\end{eqnarray}
and
\begin{eqnarray} \label{PT116-2}
&&\frac{1}{f^{n-1}(s)\theta(s,\xi)e^{-\phi}}\frac{d}{ds}\left[f^{n-1}(s)\theta(s,\xi)e^{-\phi}\frac{d\psi(s)}{ds}\right] \nonumber\\
 &&\quad = \frac{d^{2}\psi}{ds^{2}}+\left[(n-1)\frac{f'(s)}{f(s)}-\phi'(s)+\frac{d\theta(s,\xi)}{ds}\cdot\frac{1}{\theta(s,\xi)}\right]\frac{d\psi}{ds}\nonumber\\
 &&\quad =\frac{d^{2}\psi}{ds^{2}}+
\left\{(n-1)\frac{f'(s)}{f(s)}-\phi'(s)+(n-1)\frac{f(s)}{J(s,\xi)}\left[\frac{J(s,\xi)}{f(s)}\right]'\right\}\frac{d\psi}{ds}.
\end{eqnarray}
By Lemma \ref{lemma6-1}, $\psi'(s)<0$ for $0<s<r_{0}$. Together with
(\ref{6-1}), one has from (\ref{PT116-2}) that
 \begin{eqnarray} \label{PT116-3}
&&\frac{1}{f^{n-1}(s)\theta(s,\xi)e^{-\phi}}\frac{d}{ds}\left[f^{n-1}(s)\theta(s,\xi)e^{-\phi}\frac{d\psi(s)}{ds}\right]\nonumber\\
&&\qquad\qquad \qquad\qquad \qquad =-\lambda_{1,\phi}\left(\mathscr{B}_{n}(q^{-},r_0)\right)\psi(s)\nonumber\\
&&\qquad\qquad\qquad\qquad\qquad\qquad\qquad
 +
(n-1)\frac{f(s)}{J(s,\xi)}\left[\frac{J(s,\xi)}{f(s)}\right]'\frac{d\psi}{ds}
\nonumber\\
&&\qquad\qquad \qquad\qquad\qquad \geq
-\lambda_{1,\phi}\left(\mathscr{B}_{n}(q^{-},r_0)\right)\psi(s)
\qquad
 \end{eqnarray}
by using the fact
$d\theta(s,\xi)/ds=(n-1)\left[J(s,\xi)/f(s)\right]^{n-2}\left[J(s,\xi)/f(s)\right]'\leq0$
shown in Theorem \ref{BishopI}. Substituting (\ref{PT116-3}) into
(\ref{PT116-1}) yields
\begin{eqnarray} \label{PT116-4}
 \int\limits_{0}^{a(\xi)}\left(\frac{d\psi}{ds}\right)^{2}f^{n-1}(s)\theta(s,\xi)e^{-\phi}ds=
 \psi(a(\xi))\psi'(a(\xi))f^{n-1}(a(\xi))\theta(a(\xi),\xi)e^{-\phi(a(\xi))}\nonumber\\
 +
 \int\limits_{0}^{a(\xi)}\psi^{2}(s)\lambda_{1,\phi}\left(\mathscr{B}_{n}(q^{-},r_0)\right)\cdot
 f^{n-1}(s)\theta(s,\xi)e^{-\phi}ds\nonumber\\
 \leq \int\limits_{0}^{a(\xi)}\psi^{2}(s)\lambda_{1,\phi}\left(\mathscr{B}_{n}(q^{-},r_0)\right)\cdot
 f^{n-1}(s)\theta(s,\xi)e^{-\phi}ds,
 \end{eqnarray}
 where the facts $\psi'(0)=0$ and $\psi(a(\xi))\psi'(a(\xi))\leq0$ have been used.
 Therefore, it follows from (\ref{PT116-4}) that
  \begin{eqnarray*}
&& \int\limits_{\xi\in
 S_{q}^{n-1}}\left[\int\limits_{0}^{a(\xi)}\left(\frac{d\psi}{ds}\right)^{2}f^{n-1}(s)\theta(s,\xi)e^{-\phi}ds\right]d\sigma
 \leq \\
&&\qquad\qquad\qquad \int\limits_{\xi\in
 S_{q}^{n-1}}\left[\int\limits_{0}^{a(\xi)}\psi^{2}(s)\lambda_{1,\phi}\left(\mathscr{B}_{n}(q^{-},r_0)\right)\cdot
 f^{n-1}(s)\theta(s,\xi)e^{-\phi}ds\right],
\end{eqnarray*}
which directly implies
$\lambda_{1,\phi}(B(q,r_0))\leq\lambda_{1,\phi}\left(\mathscr{B}_{n}(q^{-},r_0)\right)$.
When equality holds, one has $a(\xi)=r_{0}$ for almost all $\xi\in
S^{n-1}_{q}\equiv\mathbb{S}^{n-1}$. Hence, $a(\xi)\equiv r_{0}$ for
all $\xi$, and from which we have $J(s,\xi)=f(s)$, and then
$\mathrm{vol}(B(q,r_0))=\mathrm{vol}(\mathscr{B}_n(q^{-},r_0))$.
This implies from Theorem \ref{BishopI} that $B(q,r_0)$ is isometric
to $\mathscr{B}_n(q^{-},r_0)$.
\end{proof}

Now, we wish to prove Theorem \ref{theo-2}. However, we need first
the following fact, which can be seen as a weighted version of
Barta's lemma (see \cite{BJ}).

\begin{lemma}  \label{lemma6-3}
Let $\Omega$ be a normal domain in a Riemannian manifold with
density $e^{-\phi}$, where $\phi$ is a smooth real-valued function
defined on this manifold, and let $h\in C^{2}(\Omega)\cap
C^{0}(\bar{\Omega})$, with $h|_{\Omega}>0$ and
$h|_{\partial\Omega}=0$. Then
 \begin{eqnarray*}
 \inf_{\Omega}\left(\frac{\Delta_{\phi}h}{h}\right)\leq-\lambda_{1,\phi}(\Omega)\leq\sup_{\Omega}\left(\frac{\Delta_{\phi}h}{h}\right),
 \end{eqnarray*}
  where $\lambda_{1,\phi}(\Omega)$ denotes the lowest Dirichlet eigenvalue of the Witten-Laplacian on the
  domain $\Omega$.
\end{lemma}

\begin{proof}
Let $\Psi$ be an eigenfunction corresponding to
$\lambda_{1,\phi}(\Omega)$, with $\Psi|_{\Omega}>0$ and
$\Psi|_{\partial\Omega}=0$. This choice of the eigenfunction $\Psi$
can be assured by \cite[Theorem 1.2]{CMW}, i.e. the number of nodal
domains of the eigenfunction $\Psi$ is $1$, which implies $\Psi$
would not change sign on $\Omega$. Set
 \begin{eqnarray*}
\mathcal {F}:=\Psi-h,
 \end{eqnarray*}
and then it follows from the fact
$\Delta_{\phi}\Psi+\lambda_{1,\phi}\Psi=0$ that
 \begin{eqnarray} \label{Plemma6-2-1}
-\lambda_{1,\phi}(\Omega)&=&-\lambda_{1,\phi}=\frac{\Delta_{\phi}\Psi}{\Psi}=\frac{\Delta
\Psi}{\Psi}
-\frac{\langle\nabla\phi,\nabla\Psi\rangle}{\Psi}\nonumber\\
&=&\frac{\Delta(\mathcal{F}+h)}{\mathcal{F}+h}-\frac{\langle\nabla\phi,\nabla(\mathcal{F}+h)\rangle}{\mathcal{F}+h}\nonumber\\
&=& \frac{\Delta h}{h}+\frac{h\Delta\mathcal{F}-\mathcal{F}\Delta
h}{h(h+\mathcal{F})}-\frac{\langle\nabla\phi,\nabla h\rangle}{h}+
\frac{\mathcal{F}\langle\nabla\phi,\nabla h\rangle-h\langle\nabla\phi,\nabla\mathcal{F}\rangle}{h(h+\mathcal{F})}\nonumber\\
&=&\frac{\Delta_{\phi}h}{h}+\frac{h\Delta_{\phi}\mathcal{F}-\mathcal{F}\Delta_{\phi}h}{h(\mathcal{F}+h)}.
 \end{eqnarray}
 Since $h|_{\Omega}>0$,
$h|_{\partial\Omega}=0$, $h(\mathcal{F}+h)|_{\Omega}>0$ and
$h(\mathcal{F}+h)|_{\partial\Omega}=0$, it is not hard to know
\begin{eqnarray*}
\int_{\Omega}(h\Delta_{\phi}\mathcal{F}-\mathcal{F}\Delta_{\phi}h)e^{-\phi}dv=0,
\end{eqnarray*}
which implies that in the domain $\Omega$, the function
$h\Delta_{\phi}\mathcal{F}-\mathcal{F}\Delta_{\phi}h$ would have
both positive and negative values at different parts. Then it is
easy to know from (\ref{Plemma6-2-1}) that
$$-\lambda_{1,\phi}(\Omega)\leq\sup_{\Omega}\left(\frac{\Delta_{\phi}h}{h}\right)$$
holds if $h\Delta_{\phi}\mathcal{F}-\mathcal{F}\Delta_{\phi}h$ takes
negative values, while
$$ \inf_{\Omega}\left(\frac{\Delta_{\phi}h}{h}\right)\leq-\lambda_{1,\phi}(\Omega)$$
holds if $h\Delta_{\phi}\mathcal{F}-\mathcal{F}\Delta_{\phi}h$ takes
positive values. This completes the proof of Lemma \ref{lemma6-3}.
\end{proof}

\begin{remark}
\rm{ From the above proof of Lemma \ref{lemma6-3}, one easily knows
that its assertion is valid without requiring that the potential
function $\phi$ has \textbf{Property 1}.}
\end{remark}

\begin{proof}[Proof of Theorem \ref{theo-2}]
We shall use a similar argument to that in the proof of
\cite[Theorem 4.4]{fmi}. Let $\psi(s):[0,r_0]\rightarrow[0,\infty)$
 be a nonnegative
radial eigenfunction of
$\lambda_{1,\phi}\left(\mathscr{B}_{n}(q^{+},r_0)\right)$ of
$M^{+}$. By Lemma \ref{lemma6-1}, one knows that the eigenfunction
$\psi$ satisfies (\ref{6-1}), with $\psi'(0)=\psi(r_0)=0$,
$\psi\geq0$ on $[0,r_0)$, and $\frac{d\psi}{ds}<0$ on $(0,r_{0})$.
Define a function
$\mathcal{F}:\overline{B(q,r_{0})}\rightarrow[0,\infty)$ by
\begin{eqnarray*}
\mathcal{F}(\exp_{q}(s\xi))=\psi(s)
\end{eqnarray*}
for $(s,\xi)\in[0,r_{0})\times S_{q}^{n-1}$. Then by a
straightforward calculation in geodesic spherical coordinates around
$q$ and using the characterization (\ref{chr-1}), one can get
 \begin{eqnarray*}
\frac{\Delta_{\phi}\mathcal{F}}{\mathcal{F}}(\exp_{q}(s\xi))\leq\frac{1}{\psi}\left[\psi''(s)+(n-1)\frac{f'(s)}{f(s)}\psi'(s)-\phi'(s)\psi'(s)\right]
=-\lambda_{1,\phi}\left(\mathscr{B}_{n}(q^{+},r_0)\right),
 \end{eqnarray*}
which, together with Lemma \ref{lemma6-3}, implies
\begin{eqnarray*}
-\lambda_{1,\phi}(B(q,r))\leq\sup\frac{\Delta_{\phi}\mathcal{F}}{\mathcal{F}}(\exp_{q}(s\xi))\leq-\lambda_{1,\phi}\left(\mathscr{B}_{n}(q^{+},r_0)\right),
\end{eqnarray*}
i.e.
$\lambda_{1,\phi}(B(q,r))\geq\lambda_{1,\phi}\left(\mathscr{B}_{n}(q^{+},r_0)\right)$.
When the equality holds, the isometric rigidity between $B(q,r_0)$
and $\mathscr{B}_{n}(q^{+},r_0)$ can be obtained by directly using
Theorem \ref{BishopII}.
\end{proof}

As inspired by \cite{JM4}, we found that the assertions in Theorem
\ref{theo-1} can be improved to the case of weighted $p$-Laplacian
$\Delta_{p,\phi}$, $1<p<\infty$. In order to get that improvement,
we need first the following fact.

\begin{lemma} \label{lemma6-5}
Assume that the potential function $\phi$ has \textbf{Property 1}.
The eigenfunction corresponding to the first Dirichlet eigenvalue
 of the
 weighted $p$-Laplacian ($1<p<\infty$) on the geodesic ball $\mathscr{B}_{n}(q^{-},r_0)$  may be chosen to be nonnegative and is
a radial function $\psi(s)$ satisfying $\psi'(s)<0$ for $0<s\leq
r_{0} <\min\{\ell(q),l\}$, with $\ell(q)$ defined as in
(\ref{key-def1}). The conclusion is still valid for the
eigenfunction belonging to the first Dirichlet eigenvalue
 of the
 weighted $p$-Laplacian on the geodesic ball
 $\mathscr{B}_{n}(q^{+},r_0)$ for $0<s\leq r_{0}
<\min\{\mathrm{inj}(q),l\}$, with $\mathrm{inj}(q)$ defined by
(\ref{inj-R}).
\end{lemma}

\begin{proof}
As mentioned at the end of Section \ref{Sect2}, by making suitable
adjustments to the arguments in \cite{AA, AT, BK, GAPA, LP} and
using the variational principle, some basic properties can be
obtained, including:
\begin{itemize}
\item \emph{The first Dirichlet eigenvalue $\lambda_{1,p}^{\phi}(\Omega)$  is simple, isolated, and eigenfunctions
associated with $\lambda_{1,p}^{\phi}(\Omega)$ do not change sign.}
\end{itemize}
Without loss of generality, one may choose that $\psi$ is
nonnegative on $\mathscr{B}_{n}(q^{-},r_0)$.  The first assertion of
Lemma \ref{lemma6-5} follows directly
 by choosing $\Omega=\mathscr{B}_{n}(q^{-},r_0)$. If furthermore $\phi$ has \textbf{Property
 1}, then the weighted $p$-Laplacian $\Delta_{p,\phi}$ on $\mathscr{B}_{n}(q^{-},r_0)$ in geodesic
spherical coordinates at $q^{-}$ is given by
 \begin{eqnarray*}
&&\Delta_{p,\phi}=|\nabla(\cdot)|^{p-2}\frac{d^{2}}{ds^{2}}+\frac{d}{ds}(|\nabla(\cdot)|^{p-2})\frac{d}{ds}+
(n-1)\frac{f'(s)}{f(s)}|\nabla(\cdot)|^{p-2}\frac{d}{ds}
+\frac{1}{f^{2}(s)}\Delta_{p, \mathbb{S}^{n-1}}\\
&&\qquad\qquad\qquad\qquad\qquad\qquad\qquad\qquad\qquad
-|\nabla(\cdot)|^{p-2}\langle\phi'(s),\nabla\cdot\rangle,
 \end{eqnarray*}
where $\Delta_{p,\mathbb{S}^{n-1}}$ denotes the $p$-Laplacian on
$\mathbb{S}^{n-1}$ with respect to the round metric. Since
$\lambda_{1,p}^{\phi}(\mathscr{B}_{n}(q^{-},r_0))$  is simple,
isolated, which implies that the dimension of the eigenspace
corresponding to
$\lambda_{1,\phi}\left(\mathscr{B}_{n}(q^{-},r_0)\right)$ is $1$, it
follows that $\psi$ is radial and satisfies
\begin{eqnarray} \label{6-8}
&&(p-1)|\psi'(s)|^{p-2}\psi''(s)+\left[(n-1)\frac{f'(s)}{f(s)}-\phi'(s)\right]|\psi'(s)|^{p-2}\psi'(s)\nonumber\\
&&\qquad\qquad\qquad\qquad\qquad\qquad\qquad
+\lambda^{\phi}_{1,p}(\mathscr{B}_{n}(q^{-},r_0))
|\psi(s)|^{p-2}\psi(s)=0
\end{eqnarray}
 and the boundary condition $\psi'(0)=0$ and $\psi(r_0)=0$. Set
$\alpha(s)=e^{\int_{\epsilon}^{s}\left[(n-1)\frac{f'(t)}{f(t)}-\phi'(t)\right]dt}$,
$\epsilon\rightarrow0^{+}$, and then it is not hard to know that
(\ref{6-8}) becomes
\begin{eqnarray} \label{6-9}
(|\psi'|^{p-2}\alpha\psi')'+\lambda^{\phi}_{1,p}\left(\mathscr{B}_{n}(q^{-},r_0)\right)\alpha\psi|\psi|^{p-2}=0.
\end{eqnarray}
Integrating both sides of (\ref{6-9}) on the interval $(0,s)$ with
$s<r_0$, one has
\begin{eqnarray*}
\alpha(s)|\psi'(s)|^{p-2}\psi'(s)=-\lambda^{\phi}_{1,p}\int_{0}^{s}\alpha(t)|\psi(t)|^{p-2}\psi(t)dt,
\end{eqnarray*}
and this implies that $\psi'(s)|_{(0,r_0]}<0$ whenever
$\psi(s)|_{(0,r_0)}>0$ and
$\lambda^{\phi}_{1,p}\left(\mathscr{B}_{n}(q^{-},r_0)\right)>0$.
\end{proof}

We have the following Cheng-type eigenvalue comparison theorem of
the weighted $p$-Laplacian.

\begin{theorem} \label{theo-3}
Assume that an $n$-dimensional ($n\geq2$) complete Riemannian
manifold $M^n$ has a radial Ricci curvature lower bound
$(n-1)\kappa(r)$ w.r.t. $q\in M^n$, where $r=d_{M^n}(q,\cdot)$
denotes the Riemannian distance from $q$. Assume that the potential
function $\phi$ has \textbf{Property 1}. Then  for $1<p<\infty$ and
$r_{0}<\min\{\ell(q),l\}$ with $\ell(q)$ defined as in
(\ref{key-def1}), one has
\begin{eqnarray} \label{ECT-7}
\lambda^{\phi}_{1,p}(B(q,r_0))\leq\lambda^{\phi}_{1,p}\left(\mathscr{B}_{n}(q^{-},r_0)\right),
\end{eqnarray}
and the equality holds if and only if $B(q,r_0)$ is isometric to
$\mathscr{B}_{n}(q^{-},r_0)$.
\end{theorem}

\begin{proof}
We wish to use a similar argument to that in the proof of
\cite[Theorem 3.2]{JM4}. Let $\psi$ be an eigenfunction of the
Dirichlet eigenvalue
$\lambda^{\phi}_{1,p}\left(\mathscr{B}_{n}(q^{-},r_0)\right)$,
which, by Lemma \ref{lemma6-5}, is simple, radial and can be chosen
to be nonnegative. Since $\psi\circ r$ vanishes on the boundary
$\partial B(q,r_0)$, where $r(x)=d_{M^{n}}(q,x)$ measures the radial
distance from $q$ to any point $x\in \partial B(q,r_0)$, it is not
hard to know that $\psi\circ r\in W^{1,p}_{0,\phi}(B(q,r_0))$ and
 \begin{eqnarray*}
 \lambda^{\phi}_{1,p}(B(q,r_0))\leq\frac{\int\limits_{B(q,r_0)}|d\psi\circ r|^{p}\cdot e^{-\phi }dv}
 {\int\limits_{B(q,r_0)}|\psi\circ r|^{p}\cdot e^{-\phi }dv}
 \end{eqnarray*}
 by using the
characterization (\ref{chr-2}).

Similar to those calculations in the proof of Theorem \ref{theo-1},
one can get, by using spherical geodesic coordinates centered at $q$
under the integrals, that
 \begin{eqnarray*}
\int\limits_{B(q,r_0)}|d\psi\circ r|^{p}\cdot e^{-\phi
}dv=\int\limits_{\xi\in
S_{q}^{n-1}}\left[\int\limits_{0}^{a(\xi)}|\psi'(s)|^{p}f^{n-1}(s)\theta(s,\xi)e^{-\phi}ds\right]d\sigma
 \end{eqnarray*}
 and
  \begin{eqnarray*}
\int\limits_{B(q,r_0)}|\psi\circ r|^{p}\cdot e^{-\phi
}dv=\int\limits_{\xi\in
S_{q}^{n-1}}\left[\int\limits_{0}^{a(\xi)}|\psi(s)|^{p}f^{n-1}(s)\theta(s,\xi)e^{-\phi}ds\right]d\sigma.
  \end{eqnarray*}
On the other hand, since $f(s)>0$ on $(0,r_0)$ and by Lemma
\ref{lemma6-5}, $\psi(s)|_{(0,r_0)}\geq0$ and
$\psi'(s)|_{(0,r_{0}]}<0$, it follows by a direct calculation that
\begin{eqnarray} \label{PT66-1}
&&\int\limits_{0}^{a(\xi)}|\psi'(s)|^{p}f^{n-1}(s)\theta(s,\xi)e^{-\phi}ds=
-\psi(s)|\psi'(s)|^{p-1}e^{-\phi}f^{n-1}(s)\theta(s,\xi)\Big{|}_{0}^{a(\xi)}+
\nonumber\\
&&
\int\limits_{0}^{a(\xi)}\frac{\psi(s)}{f^{n-1}(s)\theta(s,\xi)e^{-\phi}}\frac{d}{ds}\left[|\psi'(s)|^{p-1}f^{n-1}(s)\theta(s,\xi)e^{-\phi}\frac{d\psi(s)}{ds}\right]
f^{n-1}(s)\theta(s,\xi)e^{-\phi}ds \qquad\quad
\end{eqnarray}
and
\begin{eqnarray} \label{PT66-2}
&&\frac{1}{f^{n-1}(s)\theta(s,\xi)e^{-\phi}}\frac{d}{ds}\left[|\psi'(s)|^{p-1}f^{n-1}(s)\theta(s,\xi)e^{-\phi}\right] \nonumber\\
 && \quad= -|\psi'(s)|^{p-2}\left\{(p-1)\psi''(s)+\left[(n-1)\frac{f'(s)}{f(s)}-\phi'(s)+\frac{d\theta(s,\xi)}{ds}\cdot\frac{1}{\theta(s,\xi)}\right]\psi'(s)\right\}\nonumber\\
 && \quad=
-|\psi'(s)|^{p-2}\bigg{\{}(p-1)\psi''(s)+\left[(n-1)\frac{f'(s)}{f(s)}-\phi'(s)+(n-1)\frac{f(s)}{J(s,\xi)}\left(\frac{J(s,\xi)}{f(s)}\right)'\right]\nonumber\\
 && \qquad\qquad\qquad \qquad\qquad \cdot \psi'(s)\bigg{\}}.
\end{eqnarray}
Together with (\ref{6-8}), one has from (\ref{PT66-2}) that
\begin{eqnarray*}
\frac{1}{f^{n-1}(s)\theta(s,\xi)e^{-\phi}}\frac{d}{ds}\left[|\psi'(s)|^{p-1}f^{n-1}(s)\theta(s,\xi)e^{-\phi}\right]&=&\lambda^{\phi}_{1,p}\left(\mathscr{B}_{n}(q^{-},r_0)\right)|\psi(s)|^{p-2}\psi(s)\\
-|\psi'(s)|^{p-2}
(n-1)\frac{f(s)}{J(s,\xi)}\left[\frac{J(s,\xi)}{f(s)}\right]'\psi'(s)
\\
&\leq&
\lambda_{1,p}^{\phi}\left(\mathscr{B}_{n}(q^{-},r_0)\right)|\psi(s)|^{p-2}\psi(s)
\qquad
 \end{eqnarray*}
by using the fact
$d\theta(s,\xi)/ds=(n-1)\left[J(s,\xi)/f(s)\right]^{n-2}\left[J(s,\xi)/f(s)\right]'\leq0$
shown in Theorem \ref{BishopI}. Substituting the above inequality
into (\ref{PT66-1}) results in
\begin{eqnarray} \label{PT66-3}
 \int\limits_{0}^{a(\xi)}|\psi'(s)|^{p}f^{n-1}(s)\theta(s,\xi)e^{-\phi}ds=-
 \psi(a(\xi))|\psi'(a(\xi))|^{p-1}f^{n-1}(a(\xi))\theta(a(\xi),\xi)e^{-\phi(a(\xi))}\nonumber\\
+
 \int\limits_{0}^{a(\xi)}|\psi(s)|^{p}\lambda^{\phi}_{1,p}\left(\mathscr{B}_{n}(q^{-},r_0)\right)\cdot
 f^{n-1}(s)\theta(s,\xi)e^{-\phi}ds\nonumber\\
 \leq \int\limits_{0}^{a(\xi)}|\psi(s)|^{p}\lambda^{\phi}_{1,p}\left(\mathscr{B}_{n}(q^{-},r_0)\right)\cdot
 f^{n-1}(s)\theta(s,\xi)e^{-\phi}ds, \qquad
 \end{eqnarray}
 where the facts $\psi'(0)=0$ and $\psi(a(\xi))\geq0$ have been
 used. Therefore, it follows from (\ref{PT66-3}) that
  \begin{eqnarray*}
&& \int\limits_{\xi\in
 S_{q}^{n-1}}\left[\int\limits_{0}^{a(\xi)}|\psi'(s)|^{p}f^{n-1}(s)\theta(s,\xi)e^{-\phi}ds\right]d\sigma
 \leq \\
&&\qquad\qquad\qquad \int\limits_{\xi\in
 S_{q}^{n-1}}\left[\int\limits_{0}^{a(\xi)}|\psi(s)|^{p}\lambda_{1,\phi}\left(\mathscr{B}_{n}(q^{-},r_0)\right)\cdot
 f^{n-1}(s)\theta(s,\xi)e^{-\phi}ds\right],
\end{eqnarray*}
which directly implies
$\lambda^{\phi}_{1,p}(B(q,r_0))\leq\lambda^{\phi}_{1,p}\left(\mathscr{B}_{n}(q^{-},r_0)\right)$.
When equality holds, one has $a(\xi)=r_{0}$ for almost all $\xi\in
S^{n-1}_{q}\equiv\mathbb{S}^{n-1}$. Hence, $a(\xi)\equiv r_{0}$ for
all $\xi$, and from which we have $J(s,\xi)=f(s)$, and then
$\mathrm{vol}(B(q,r_0))=\mathrm{vol}(\mathscr{B}_n(q^{-},r_0))$.
This implies from Theorem \ref{BishopI} that $B(q,r_0)$ is isometric
to $\mathscr{B}_n(q^{-},r_0)$.
\end{proof}

\begin{remark}
\rm{ It is easy to see that if $p=2$, then the weighted
$p$-Laplacian $\Delta_{p,\phi}$ would degenerate into the
Witten-Laplacian $\Delta_{\phi}$ directly, and correspondingly,
Theorem \ref{theo-3} becomes Theorem \ref{theo-1} exactly. However,
if one checks the end of Section \ref{Sect2} carefully, he or she
would find that the spectral structure of the eigenvalue problem
(\ref{eigenp-wpl}) is much complicated than the one of the
eigenvalue problem (\ref{eigenp-w1}) --- the operator
$\Delta_{p,\phi}$ in (\ref{eigenp-wpl}) has a discrete spectrum and
it might also has an essential spectrum, which deeply depends on the
geometry and topology of the domain $\Omega$, but the operator
$\Delta_{\phi}$ in (\ref{eigenp-w1}) definitely only has a discrete
spectrum and each eigenvalue in this discrete spectrum can be
characterized clearly. }
\end{remark}

\begin{remark}
\rm{ Denote by $\lambda_{k,p}(\cdot)$ the $k$-th Dirichlet
eigenvalue of the $p$-Laplacian $\Delta_{p}$ on a given geometric
object. Clearly, if $\phi=const.$, then (\ref{ECT-7}) in Theorem
\ref{theo-3} would reduce to
\begin{eqnarray} \label{ECT-8}
\lambda_{1,p}(B(q,r_0))\leq\lambda_{1,p}\left(\mathscr{B}_{n}(q^{-},r_0)\right),
\end{eqnarray}
and the equality holds if and only if $B(q,r_0)$ is isometric to
$\mathscr{B}_{n}(q^{-},r_0)$.  This is exactly the assertions in
\cite[Theorem 3.2]{JM4}. If furthermore the curvature assumption was
strengthened to be $\mathrm{Ric}(M^n)\geq(n-1)K$ for some constant
$K$, then the eigenvalue comparison (\ref{ECT-8}) would reduce to
\begin{eqnarray} \label{ECT-9}
\lambda_{1,p}(B(q,r_0))\leq\lambda_{1,p}\left(\mathcal{B}_{n}(K,r_0)\right),
\end{eqnarray}
 and the equality holds if and only if $B(q,r_0)$ is isometric to
$\mathcal{B}_{n}(K,r_0)$. This eigenvalue comparison (\ref{ECT-9})
and related rigidity result have already been obtained separately by
A. M. Matei \cite{MAM} and H. Takeuchi \cite{TH}. }
\end{remark}

As we know, for an arbitrary point $q\in\mathbb{S}^{n-1}$, its cut
locus is a single point set and exactly consists of the antipodal
point $\widehat{q}$ of $q$. For a sufficiently small positive
constant $\epsilon$, the geodesic ball $B(q,\pi-\epsilon)$ tends to
$\mathbb{S}^{n-1}\setminus\{\widehat{q}\}$ as
$\epsilon\rightarrow0^{+}$, which leads to a fact that as
$\epsilon\rightarrow0^{+}$, the first Dirichlet eigenvalue
$\lambda_{1}(B(q,\pi-\epsilon))$ would tend to the first closed
eigenvalue of the Laplacian on $\mathbb{S}^{n-1}$, that is,
\begin{eqnarray*}
\lim\limits_{\epsilon\rightarrow0^{+}}\lambda_{1}(B(q,\pi-\epsilon))=0.
\end{eqnarray*}
This spectral asymptotical property on the unit Euclidean
$(n-1)$-sphere $\mathbb{S}^{n-1}$ has been extended to the first
Dirichlet eigenvalue of the Laplacian on ``closed" spherically
symmetric manifolds (see \cite[lemma2.5]{fmi}), and later to the
first Dirichlet eigenvalue of the $p$-Laplacian ($1<p<\infty$) on
``closed" spherically symmetric manifolds (see \cite[Lemma
2.2]{JM4}). It is natural to ask whether or not these extensions can
be further improved to the case of weighted $p$-Laplacian. In fact,
by using the same cut-off continuous functions as those two in the
proof of \cite[Lemma 2.2]{JM4}, one has:

\begin{lemma} \label{lemma6-9}
Assume that $M^{\ast}$ is a generalized space form
$[0,l)\times_{f}\mathbb{S}^{n-1}$ (with $q^{\ast}$ as its base
point) and $\phi$ is a smooth real-valued function defined on
$M^{\ast}$, where $f\in C^{2}([0,l))$ and $C^3$ at $s=0$,
$f(0)=f''(0)=0$, $f'(0)=1$, closing at $s=l$, i.e. $f(l)=0$. We
have \\
(I) in case $n=2$, if for some $\varepsilon>0$, $f\in
C^{1}([0,l+\varepsilon))$, then $\lim_{r\rightarrow
l^{-}}\lambda_{1,p}^{\phi}\left(\mathscr{B}_{n}(q^{\ast},r)\right)=0$
with $1<p\leq2$;\\
 (II) in case $n\geq3$, if for some $\varepsilon>0$, $f\in
C^{2}([0,l+\varepsilon))$, then $\lim_{r\rightarrow
l^{-}}\lambda_{1,p}^{\phi}\left(\mathscr{B}_{n}(q^{\ast},r)\right)=0$
with $1<p<3$.
\end{lemma}

\noindent The reason why the same cut-off continuous functions still
work in the weighted case is that $e^{-\phi}$ is always bounded on
the geodesic ball $\mathscr{B}_{n}(q^{\ast},r)\subseteq M^{\ast}$,
which leads to the result that all integrals appeared in the proof
of \cite[Lemma 2.2]{JM4} (i.e. corresponding to the non-weighted
case) have no essential difference with those would appear in the
weighted case.

It is interesting that nearly all the spectral results in this paper
need radial \textbf{Property 1} for the potential function, but this
assumption is not required to get the spectral asymptotical result
in Lemma \ref{lemma6-9}. This invokes us to consider the following
question:

\vspace{3mm}

\noindent \textbf{Open problem}. \emph{What kind of assumptions
should be imposed on the potential function $\phi$ for transplanting
classical spectral results of the Laplacian or the $p$-Laplacian to
the weighted case? What is the weakest assumption in a transplanting
process? }

\section{Appendix C}
\renewcommand{\thesection}{\arabic{section}}
\renewcommand{\theequation}{\thesection.\arabic{equation}}
\setcounter{equation}{0}

For a given Riemannian $n$-manifold $M^{n}$ and a point $q\in
M^{n}$, similar to the argument for the $p$-Laplacian in our
previous work \cite[Section 1]{DM} or \cite[Section 1]{MTZ}, it is
not hard to get that the limit
$\lambda_{1,p}^{\phi}(M^{n}):=\lim\limits_{r_{0}\rightarrow\infty}\lambda_{1,p}^{\phi}(B(q,r_0))$
exists and is independent of the choice of $q$. Based on this fact,
it is also not hard to know that the fundamental tone of the
weighted $p$-Laplacian $\Delta_{p,\phi}$ on $\Omega$ can be
 well-defined as follows
\begin{eqnarray} \label{chr-3}
\left(\lambda_{1,p}^{\phi}(\Omega)\right)^{\ast}:=\inf\left\{\frac{\int_{\Omega}|u|^{p}e^{-\phi}dv}{\int_{\Omega}|\nabla
u|^{p}e^{-\phi}dv}=\frac{\int_{\Omega}|u|^{p}d\mu}{\int_{\Omega}|\nabla
u|^{p}d\mu}\Big{|}u\neq0,~u\in W^{1,p}_{0,\phi}(\Omega)\right\},
\end{eqnarray}
with notations having the same meanings as before. If $\Omega$ is
unbounded, then $(\lambda_{1,p}^{\phi}(\Omega))^{\ast}$ coincides
with the infimum of the spectrum $\Sigma\subseteq[0,\infty)$ of the
unique self-adjoint extension of the weighted $p$-Laplacian
$\Delta_{p,\phi}$ acting on $C^{\infty}_{0}(\Omega)$, which is also
denoted by $\Delta_{p,\phi}$. If $\Omega$ has compact closure and
boundary $\partial\Omega$, from the characterization (\ref{chr-2})
one knows that $(\lambda_{1,p}^{\phi}(\Omega))^{\ast}$ equals the
first Dirichlet eigenvalue $\lambda_{1,p}^{\phi}(\Omega)$ of the
weighted $p$-Laplacian, i.e.
$(\lambda_{1,p}^{\phi}(\Omega))^{\ast}=\lambda_{1,p}^{\phi}(\Omega)$.

First, we need the following fact, which  can be seen as a nonlinear
attempt of Lemma \ref{lemma6-3}.

\begin{lemma} \label{lemma7-1}
Let $\Omega$ be a domain in a Riemannian manifold with density
$e^{-\phi}$, where $\phi$ is a smooth real-valued function defined
on this manifold. Assume that $h\in C^{1+\alpha}(\Omega)\cap
C^{0}(\bar{\Omega})$ with $\Delta_{p,\phi}h\in C^{0}(\Omega)$ and
$h|_{\Omega}>0$, where $\alpha\in(0,1)$. Then the fundamental tone
$(\lambda_{1,p}^{\phi}(\Omega))^{\ast}$ of the weighted
$p$-Laplacian ($1<p<\infty$) satisfies
 \begin{eqnarray} \label{PL-7-1}
\left(\lambda_{1,p}^{\phi}(\Omega)\right)^{\ast}\geq\inf\limits_{\Omega}\left\{-\frac{\Delta_{p,\phi}h}{h^{p-1}}\right\}.
 \end{eqnarray}
Moreover, if $\Omega$ is bounded, then
$(\lambda_{1,p}^{\phi}(\Omega))^{\ast}$ becomes the lowest Dirichlet
eigenvalue $\lambda_{1,p}^{\phi}(\Omega)$ of the weighted
$p$-Laplacian on $\Omega$, and the equality in (\ref{PL-7-1}) holds
if and only if $h=\beta u$ for some constant $\beta>0$, where $u\in
W^{1,p}_{0,\phi}(\Omega)$ is a nonnegative minimizer of the Rayleigh
quotient.
\end{lemma}

\begin{proof}
 Picone's classical inequality \cite[p. 18]{PM} states:
\begin{itemize}
\item If $u: \Omega\rightarrow\mathbb{R}$ and $\nu:
\Omega\rightarrow\mathbb{R}$ are two differentiable functions such
that $u\geq0$ and $\nu>0$, then in each connected component of
$\Omega$ one has
 \begin{eqnarray} \label{PL-7-2}
|\nabla u|^{2}+\frac{u^2}{\nu^2}|\nabla
\nu|^{2}-\frac{2u}{\nu}\langle\nabla u,\nabla \nu\rangle=|\nabla
u|^{2}-\left\langle\nabla
\nu,\nabla\left(\frac{u^2}{\nu}\right)\right\rangle\geq0.
 \end{eqnarray}
 Moreover, the equality holds if and only if $u$ is proportional to
 $\nu$.
\end{itemize}
It is not hard to see that the lower bound estimate
\begin{eqnarray*}
 \inf_{\Omega}\left(\frac{\Delta_{\phi}h}{h}\right)\leq-\lambda_{1,\phi}(\Omega)
 \end{eqnarray*}
in Lemma \ref{lemma6-3} can be obtained from (\ref{PL-7-2}) by
taking $u$ as a nonnegative eigenfunction of the first Dirichlet
eigenvalue $\lambda_{1,\phi}(\Omega)$ of the Witten-Laplacian,
$\nu=h$, and then integrating by parts the right hand side of the
equation (\ref{PL-7-2}) over $\Omega$. This observation was inspired
by the argument in the proof of \cite[Theorem 1.2]{SPHC} with
respect to the weighted volume density $d\mu=e^{-\phi}dv$. It is
natural to try to improve this observation to the nonlinear case.

In \cite{AH}, Picone's result (\ref{PL-7-2}) was extended to the
following:
\begin{eqnarray} \label{PL-7-3}
&&|\nabla u|^{p}+(p-1)\frac{u^p}{\nu^p}|\nabla
\nu|^{p}-p\frac{u^{p-1}}{\nu^{p-1}}|\nabla \nu|^{p-2}\langle\nabla
u,\nabla \nu\rangle=\nonumber\\
&&\qquad\qquad\qquad\qquad\qquad\qquad  |\nabla u|^{p}-|\nabla
\nu|^{p-2}\left\langle\nabla
\nu,\nabla\left(\frac{u^p}{\nu^{p-1}}\right)\right\rangle\geq0,
\end{eqnarray}
with $1<p<\infty$, where the equality in (\ref{PL-7-3}) holds if and
only if $u$ is proportional to
 $\nu$ in each connected component of
$\Omega$. It is easy to see that if $p=2$ in (\ref{PL-7-3}), then it
would degenerate into (\ref{PL-7-2}) exactly.

Let $\Psi\in C^{\infty}_{0}(\Omega)$ with $\Psi\geq0$, and choose
$h\in C^{1+\alpha}(\Omega)$ with $\Delta_{p,\phi}h\in C^{0}(\Omega)$
such that $h>0$ in $\Omega$. Since $\Psi$ has a compact support
$\mathrm{supp}(\Psi)$ in $\Omega$ and $h$ is continuous and strictly
positive in $\Omega$, there exists some constant $c>0$ such that
$\Psi\geq c$ on $\mathrm{supp}(\Psi)\subset\Omega$, and
$\Psi^{p}/h^{p-1}$ is well-defined and is contained in
$W^{1,p}_{0,\phi}(\Omega)$. So, $\Psi^{p}/h^{p-1}$ can be used as a
trial function in (\ref{chr-3}). By choosing $u=\Psi$, $\nu=h$ in
(\ref{PL-7-3}), and integrating over $\Omega$ with respect to the
weighted volume density $d\mu=e^{-\phi}dv$, we can obtain
 \begin{eqnarray}  \label{PL-7-4}
0&\leq&\int_{\Omega}\left(|\nabla\Psi|^{p}+(p-1)\frac{\Psi^p}{h^p}|\nabla
h|^{p}-p\frac{\Psi^{p-1}}{h^{p-1}}|\nabla h|^{p-2}\langle\nabla
\Psi,\nabla h\rangle\right)d\mu\nonumber\\
&=&\int_{\Omega}\left(|\nabla\Psi|^{p}+(p-1)\frac{\Psi^p}{h^p}|\nabla
h|^{p}-p\frac{\Psi^{p-1}}{h^{p-1}}|\nabla h|^{p-2}\langle\nabla
\Psi,\nabla h\rangle\right)e^{-\phi}dv\nonumber\\
&=&\int_{\Omega}\left[|\nabla \Psi|^{p}-|\nabla
h|^{p-2}\left\langle\nabla
h,\nabla\left(\frac{\Psi^p}{h^{p-1}}\right)\right\rangle\right]e^{-\phi}dv\nonumber\\
&=&\int_{\Omega}|\nabla\Psi|^{p}e^{-\phi}dv-\int_{\Omega}\left\langle|\nabla
h|^{p-2}\nabla
h,\nabla\left(\frac{\Psi^p}{h^{p-1}}\right)\right\rangle e^{-\phi}dv\nonumber\\
&=&\int_{\Omega}|\nabla\Psi|^{p}e^{-\phi}dv+\int_{\Omega}\mathrm{div}\left(e^{-\phi}|\nabla
h|^{p-2}\nabla h\right)\cdot\frac{\Psi^p}{h^{p-1}}dv\nonumber\\
&=&\int_{\Omega}|\nabla\Psi|^{p}e^{-\phi}dv+\int_{\Omega}(\Delta_{p,\phi}h)\cdot\frac{\Psi^p}{h^{p-1}}e^{-\phi}dv\nonumber\\
&\leq&\int_{\Omega}|\nabla\Psi|^{p}e^{-\phi}dv-\inf\limits_{\Omega}\left\{-\frac{\Delta_{p,\phi}h}{h^{p-1}}\right\}\int_{\Omega}\Psi^{p}e^{-\phi}dv,
 \end{eqnarray}
which implies
 \begin{eqnarray} \label{PL-7-5}
\inf\limits_{\Omega}\left\{-\frac{\Delta_{p,\phi}h}{h^{p-1}}\right\}\leq\frac{\int_{\Omega}|\nabla\Psi|^{p}e^{-\phi}dv}{\int_{\Omega}\Psi^{p}e^{-\phi}dv}=
\frac{\int_{\Omega}|\nabla\Psi|^{p}d\mu}{\int_{\Omega}\Psi^{p}d\mu}
 \end{eqnarray}
with  $\Psi\in C^{\infty}_{0}(\Omega)$ and $\Psi\geq0$. Taking the
infimum for the Rayleigh quotient
$$\frac{\int_{\Omega}|\nabla\Psi|^{p}e^{-\phi}dv}{\int_{\Omega}\Psi^{p}e^{-\phi}dv}$$
over $\Psi\in C^{\infty}_{0}(\Omega)\setminus\{0\}$ and $\Psi\geq0$,
together with the fact that $C^{\infty}_{0}(\Omega)$ is dense in the
Sobolev space $W^{1,p}_{0,\phi}$, we can get from (\ref{chr-3}) and
(\ref{PL-7-5}) that the estimate (\ref{PL-7-1}) in Lemma
\ref{lemma7-1} holds.

Assume that $\Omega$ is bounded, then one has first that
$(\lambda_{1,p}^{\phi}(\Omega))^{\ast}=\lambda_{1,p}^{\phi}(\Omega)$.
Correspondingly, the estimate (\ref{PL-7-1}) becomes
 \begin{eqnarray} \label{PL-7-6}
\lambda_{1,p}^{\phi}(\Omega)\geq\inf\limits_{\Omega}\left\{-\frac{\Delta_{p,\phi}h}{h^{p-1}}\right\}.
 \end{eqnarray}
 Now, we wish to give a rigidity characterization for the equality
 in (\ref{PL-7-6}). To get this purpose, we shall use a similar
 argument to that in the proof of
\cite[Theorem 1.2]{SPHC}.

We first show the validity of the sufficient condition. If $h=\beta
u$ for some constant $\beta>0$, where  $u\in
W^{1,p}_{0,\phi}(\Omega)$ is a nonnegative minimizer of the Rayleigh
quotient, that is,
\begin{eqnarray*}
\lambda_{1,p}^{\phi}(\Omega)=\frac{\int_{\Omega}|\nabla
u|^{p}e^{-\phi}dv}{\int_{\Omega}|u|^{p}e^{-\phi}dv}
\end{eqnarray*}
from the characterization (\ref{chr-2}), then one has
 \begin{eqnarray*}
\Delta_{p,\phi}(\beta u)=e^{\phi}\mathrm{div}(e^{-\phi}|\nabla(\beta
u)|^{p-2}\nabla(\beta
u))=\beta^{p-1}e^{\phi}\mathrm{div}(e^{-\phi}|\nabla u|^{p-2}\nabla
u)=\beta^{p-1}\Delta_{p,\phi}u,
 \end{eqnarray*}
 and then in the sense of weak derivatives,
  \begin{eqnarray*}
-\Delta_{p,\phi}h=-\Delta_{p,\phi}(\beta
u)=-\beta^{p-1}\Delta_{p,\phi}u=\beta^{p-1}\lambda_{1,p}^{\phi}(\Omega)u^{p-1}=\lambda_{1,p}^{\phi}(\Omega)h^{p-1}
 \end{eqnarray*}
holds in $\Omega$. This implies
 \begin{eqnarray*}
\lambda_{1,p}^{\phi}(\Omega)=-\frac{\Delta_{p,\phi}h}{h^{p-1}}=\inf\limits_{\Omega}\left\{-\frac{\Delta_{p,\phi}h}{h^{p-1}}\right\}
 \end{eqnarray*}
 with $h=\beta u$ for some constant $\beta>0$.

At the end, we show the validity of the necessary condition. Since
$\Omega$ is bounded, the infimum
$$\inf\left\{\frac{\int_{\Omega}|u|^{p}e^{-\phi}dv}{\int_{\Omega}|\nabla
u|^{p}e^{-\phi}dv}\Big{|}u\neq0,~u\in
W^{1,p}_{0,\phi}(\Omega)\right\}$$ can be achieved by some
nonnegative function $u\in W^{1,p}_{0,\phi}(\Omega)$ which is a weak
solution to the eigenvalue problem (\ref{eigenp-wpl}). In fact, $u$
is an eigenfunction of the first Dirichlet eigenvalue
$\lambda_{1,p}^{\phi}(\Omega)$, and it does not change sign in
$\Omega$ (see the end of Section \ref{Sect2}). Let
$\{\Psi_{n}\}_{n=1}^{\infty}$ be a sequence in the space
$C^{\infty}_{0}(\Omega)$ such that $\Psi_{n}\geq0$ and
$\Psi_{n}\rightarrow u$ under the norm $\|\cdot\|^{\phi}_{1,p}$ as
$n\rightarrow\infty$. Choosing $u=\Psi_{n}$, $\nu=h$ in
(\ref{PL-7-3}), and using a similar calculating process as that in
(\ref{PL-7-4}), together with the fact $\Psi_{n}^{p}/h^{p-1}\in
W^{1,p}_{0,\phi}(\Omega)$, one has
\begin{eqnarray*}
\int_{\Omega}|\nabla\Psi_{n}|^{p}e^{-\phi}dv-\int_{\Omega}\left(-\frac{\Delta_{p,\phi}h}{h^{p-1}}\right)\Psi_{n}^{p}\cdot
e^{-\phi}dv\geq0.
\end{eqnarray*}
Letting $n\rightarrow\infty$ in the above inequality yields
\begin{eqnarray}  \label{PL-7-7}
0\leq\int_{\Omega}|\nabla
u|^{p}e^{-\phi}dv-\int_{\Omega}\left(-\frac{\Delta_{p,\phi}h}{h^{p-1}}\right)u^{p}
e^{-\phi}dv=\int_{\Omega}|\nabla
u|^{p}d\mu-\int_{\Omega}\left(-\frac{\Delta_{p,\phi}h}{h^{p-1}}\right)|u|^{p}d\mu.
\qquad
\end{eqnarray}
It follows from (\ref{PL-7-7}) that
 \begin{eqnarray} \label{PL-7-8}
\inf\limits_{\Omega}\left\{-\frac{\Delta_{p,\phi}h}{h^{p-1}}\right\}\leq\frac{\int_{\Omega}|\nabla
u|^{p}d\mu}{\int_{\Omega}|u|^{p}d\mu}.
 \end{eqnarray}
Since $u\in W^{1,p}_{0,\phi}(\Omega)$ is a weak solution to the
eigenvalue problem (\ref{eigenp-wpl}), we can obtain from
(\ref{chr-2}) that
\begin{eqnarray*}
\frac{\int_{\Omega}|u|^{p}e^{-\phi}dv}{\int_{\Omega}|\nabla
u|^{p}e^{-\phi}dv}=\frac{\int_{\Omega}|\nabla
u|^{p}d\mu}{\int_{\Omega}|u|^{p}d\mu}=\lambda_{1,p}^{\phi}(\Omega).
 \end{eqnarray*}
Assume that the equality in (\ref{PL-7-1}) holds, i.e. the necessary
condition holds, which means
\begin{eqnarray*}
\lambda_{1,p}^{\phi}(\Omega)=\inf\limits_{\Omega}\left\{-\frac{\Delta_{p,\phi}h}{h^{p-1}}\right\}.
\end{eqnarray*}
Then one gets that the equalities in (\ref{PL-7-7}) and
 (\ref{PL-7-8}) hold simultaneously. Therefore, it follows from the
 fact
 \begin{eqnarray*}
\int_{\Omega}\left(-\frac{\Delta_{p,\phi}h}{h^{p-1}}\right)|u|^{p}e^{-\phi}dv=\int_{\Omega}|\nabla
h|^{p-2}\left\langle\nabla
h,\nabla\left(\frac{u^p}{h^{p-1}}\right)\right\rangle e^{-\phi}dv
 \end{eqnarray*}
that the nonnegative
 integrand
 \begin{eqnarray*}
|\nabla u|^{p}-|\nabla h|^{p-2}\left\langle\nabla
h,\nabla\left(\frac{u^p}{h^{p-1}}\right)\right\rangle
 \end{eqnarray*}
 should vanish a.e. in $\Omega$, that is to say,
 \begin{eqnarray} \label{PL-7-9}
|\nabla u|^{p}+(p-1)\frac{u^p}{h^p}|\nabla
h|^{p}-p\frac{u^{p-1}}{h^{p-1}}|\nabla h|^{p-2}\langle\nabla
u,\nabla h\rangle=0
 \end{eqnarray}
holds a.e. in $\Omega$. Set
\begin{eqnarray*}
\mathcal {X}(x):=\left[|\nabla
u|^{p}-p\frac{u^{p-1}}{h^{p-1}}|\nabla u||\nabla h|^{p-2}|\nabla
h|+(p-1)\frac{u^p}{h^p}|\nabla h|^{p}\right](x), \qquad x\in\Omega
\end{eqnarray*}
and
 \begin{eqnarray*}
\mathcal {Y}(x):=\left(p\frac{u^{p-1}}{h^{p-1}}|\nabla u||\nabla
h|^{p-2}|\nabla h|-p\frac{u^{p-1}}{h^{p-1}}|\nabla
h|^{p-2}\langle\nabla u,\nabla h\rangle\right)(x), \qquad
x\in\Omega.
 \end{eqnarray*}
Therefore, the equation (\ref{PL-7-9}) can be rewritten as
\begin{eqnarray} \label{PL-7-10}
\mathcal {X}(x)+\mathcal {Y}(x)=0 \quad\mathrm{a.e.~in}~\Omega.
\end{eqnarray}
Using Young's inequality, one has
\begin{eqnarray*}
|\nabla u|^{p}+(p-1)\frac{u^p}{h^p}|\nabla
h|^{p}&=&p\left[\frac{|\nabla u|^{p}}{p}+\frac{\left((u|\nabla
h|/h)^{p-1}\right)^{p/(p-1)}}{\frac{p}{p-1}}\right]\nonumber\\
&\geq&p|\nabla u|\left(\frac{u|\nabla h|}{h}\right)^{p-1},
\end{eqnarray*}
which implies $\mathcal{X}\geq0$ in $\Omega$. By a direct
calculation, one can get
\begin{eqnarray*}
p\frac{u^{p-1}}{h^{p-1}}|\nabla u||\nabla h|^{p-2}|\nabla
h|-p\frac{u^{p-1}}{h^{p-1}}|\nabla h|^{p-2}\langle\nabla u,\nabla
h\rangle&=&p\frac{u^{p-1}}{h^{p-1}}|\nabla h|^{p-2}\left(|\nabla
u||\nabla h|-\langle\nabla u,\nabla h\rangle\right) \nonumber\\
&\geq&0,
\end{eqnarray*}
where the last inequality holds because of the Cauchy-Schwarz
inequality. So, $\mathcal{Y}\geq0$ in $\Omega$. Together with
(\ref{PL-7-10}), one easily knows that
\begin{eqnarray*}
\mathcal{X}(x)=\mathcal{Y}(x)=0  \quad\mathrm{a.e.~in}~\Omega.
\end{eqnarray*}
From $\mathcal{Y}=0$, one has
\begin{eqnarray*}
|\nabla u||\nabla h|=\langle\nabla u,\nabla h\rangle,
\end{eqnarray*}
which implies that vector fields $\nabla u$ and $\nabla h$ have the
same direction a.e. in $\Omega$. That is to say, there exists a
nonnegative measurable function $\chi:\Omega\rightarrow[0,\infty)$
such that
\begin{eqnarray}  \label{PL-7-add}
\nabla h=\chi\cdot\nabla u
\end{eqnarray}
a.e. in $\Omega$. Set $\mathcal{D}:=\{x\in\Omega||\nabla u||\nabla
h|>0\}$. Since $\mathcal{X}(x)=0$ a.e. in $\Omega$, for a.e.
$x\in\mathcal{D}$ one has
\begin{eqnarray*}
\left(\frac{h}{u|\nabla h|}\right)^{p}\cdot\left[|\nabla
u|^{p}-p\frac{u^{p-1}}{h^{p-1}}|\nabla u||\nabla h|^{p-2}|\nabla
h|+(p-1)\frac{u^p}{h^p}|\nabla h|^{p}\right]=0,
\end{eqnarray*}
that is,
\begin{eqnarray} \label{PL-7-11}
\left(\frac{h}{u}\frac{|\nabla u|}{|\nabla
h|}\right)^{p}-p\frac{h}{u}\frac{|\nabla u|}{|\nabla h|}+(p-1)=0
\quad\mathrm{a.e.~in}~\mathcal{D}.
\end{eqnarray}
Set
$$\mathcal{A}:=\frac{h}{u}\frac{|\nabla u|}{|\nabla
h|},$$
 and it is easy to see $\mathcal{A}$ is positive in $\mathcal{D}$. Then (\ref{PL-7-11})
becomes $\mathcal{A}^{p}-p\mathcal{A}+(p-1)=0$, which is an equation
of the variable $\mathcal{A}$ of order $p>1$. It is clear that
$\mathcal{A}=1$ is a solution to the equation
$\mathcal{A}^{p}-p\mathcal{A}+(p-1)=0$. However, one can see that
the function $\mathcal{A}\rightarrow\mathcal{A}^{p}$ is strictly
convex on $(0,\infty)$, which implies that $\mathcal{A}=1$ would be
the unique solution. Hence, one has $\mathcal{A}=1$ a.e. in
$\mathcal{D}$, i.e.
\begin{eqnarray} \label{PL-7-12}
h|\nabla u|=u|\nabla h| \quad\mathrm{a.e.~in}~\mathcal{D}.
\end{eqnarray}
On $\Omega\setminus\mathcal{D}$, $|\nabla u||\nabla h|=0$. If there
exists $x\in\Omega\setminus\mathcal{D}$ such that $|\nabla h(x)|=0$,
then it follows from (\ref{PL-7-9}) that $|\nabla u(x)|^{p}=0$.
Similarly, if there exists $x\in\Omega\setminus\mathcal{D}$ such
that $|\nabla u(x)|=0$, then it follows from (\ref{PL-7-9}) that
$(p-1)\frac{u^p(x)}{h^p(x)}|\nabla h(x)|^{p}=0$, which implies
$|\nabla h(x)|=0$. Hence, one has
\begin{eqnarray} \label{PL-7-13}
|\nabla u|=|\nabla h|=0
\quad\mathrm{a.e.~in}~\Omega\setminus\mathcal{D}.
\end{eqnarray}
Combining (\ref{PL-7-add}), (\ref{PL-7-12}) and (\ref{PL-7-13}), we
can obtain
\begin{eqnarray*}
\chi=\frac{h}{u}.
\end{eqnarray*}
Since $u, h>0$ in $\Omega$ and $u,h\in
C^{1,\alpha}_{\mathrm{loc}}(\Omega)$, one has
\begin{eqnarray*}
\nabla(\log h-\log u)=\frac{\nabla h}{h}-\frac{\nabla
u}{u}=\frac{\frac{h}{u}\cdot\nabla u}{h}-\frac{\nabla u}{u}=0,
\end{eqnarray*}
which implies that $\log h-\log u$ is $u$ is constant on each
connected component of $\Omega$. Consequently, $\chi=h/u=\beta$ for
some constant $\beta>0$ should be a constant function. Our proof is
finished.
\end{proof}

\begin{remark}
\rm{From the above proof of Lemma \ref{lemma7-1}, one easily knows
that its assertion is valid without requiring that the potential
function $\phi$ has \textbf{Property 1}. }
\end{remark}

Now, we can prove:

\begin{theorem}  \label{theo-4}
Assume that an $n$-dimensional ($n\geq2$) complete Riemannian
manifold $M^n$ has a radial sectional curvature upper bound
$\kappa(s)$ w.r.t. $q\in M^n$, and assume that the potential
function $\phi$ has \textbf{Property 1}. Then for
$r_{0}<\min\{\mathrm{inj}(q),l\}$ with $\mathrm{inj}(q)$ defined as
in (\ref{inj-R}), one has
\begin{eqnarray}  \label{ECT-9}
\lambda^{\phi}_{1,p}(B(q,r_0))\geq\lambda^{\phi}_{1,p}\left(\mathscr{B}_{n}(q^{+},r_0)\right),
\end{eqnarray}
and the equality holds if and only if $B(q,r_0)$ is isometric to
$\mathscr{B}_{n}(q^{+},r_0)$.
\end{theorem}

\begin{proof}
Let $\psi(s):[0,r_0]\rightarrow[0,\infty)$
 be a nonnegative
radial eigenfunction of
$\lambda_{1,\phi}\left(\mathscr{B}_{n}(q^{+},r_0)\right)$ of
$M^{+}$. By Lemma \ref{lemma6-5}, one knows that the eigenfunction
$\psi$ satisfies (\ref{6-8}), with $\psi'(0)=\psi(r_0)=0$,
$\psi\geq0$ on $[0,r_0)$, and $\frac{d\psi}{ds}<0$ on $(0,r_{0})$.
Define a function
$\mathcal{F}:\overline{B(q,r_{0})}\rightarrow[0,\infty)$ by
\begin{eqnarray*}
\mathcal{F}(\exp_{q}(s\xi))=\psi(s)
\end{eqnarray*}
for $(s,\xi)\in[0,r_{0})\times S_{q}^{n-1}$. Then by a
straightforward calculation in geodesic spherical coordinates around
$q$ and using the characterization (\ref{chr-2}), one can get
\begin{eqnarray*}
\frac{\Delta_{p,\phi}\mathcal{F}}{\mathcal{F}^{p-1}}(\exp_{q}(s\xi))&\leq&\frac{1}{\psi^{p-1}}
\left\{(p-1)|\psi'(s)|^{p-2}\psi''(s)+\left[(n-1)\frac{f'(s)}{f(s)}-\phi'(s)\right]|\psi'(s)|^{p-2}\psi'(s)\right\} \\
 &=&-\lambda^{\phi}_{1,p}\left(\mathscr{B}_{n}(q^{+},r_0)\right),
\end{eqnarray*}
which implies
\begin{eqnarray*}
\inf\limits_{\Omega}\left\{-\frac{\Delta_{p,\phi}\mathcal{F}}{\mathcal{F}^{p-1}}\right\}\geq\lambda^{\phi}_{1,p}\left(\mathscr{B}_{n}(q^{+},r_0)\right).
\end{eqnarray*}
Applying Lemma \ref{lemma7-1}, one has
\begin{eqnarray*}
\lambda^{\phi}_{1,p}(B(q,r_0))\geq\inf\limits_{B(q,r_0)}\left\{-\frac{\Delta_{p,\phi}\mathcal{F}}{\mathcal{F}^{p-1}}\right\}\geq\lambda^{\phi}_{1,p}\left(\mathscr{B}_{n}(q^{+},r_0)\right),
\end{eqnarray*}
which implies (\ref{ECT-9}) directly. When the equality in
(\ref{ECT-9}) holds, the isometric rigidity between $B(q,r_0)$ and
$\mathscr{B}_{n}(q^{+},r_0)$ can be obtained by directly using
Theorem \ref{BishopII}.
\end{proof}

\begin{remark}
\rm{ Clearly, Theorem \ref{theo-2} exactly corresponds to the case
$p=2$ of Theorem \ref{theo-4}. }
\end{remark}

\begin{remark}
\rm{ If $\phi=const.$, then (\ref{ECT-9}) in Theorem \ref{theo-4}
would reduce to
\begin{eqnarray}  \label{ECT-10}
\lambda_{1,p}(B(q,r_0))\geq\lambda_{1,p}\left(\mathscr{B}_{n}(q^{+},r_0)\right),
\end{eqnarray}
and the equality holds if and only if $B(q,r_0)$ is isometric to
$\mathscr{B}_{n}(q^{+},r_0)$. If furthermore the curvature
assumption was strengthened to be that the radial sectional
curvature $\mathscr{K}(\cdot,\cdot)$ has an upper bound $K$ for some
constant $K$, then the eigenvalue comparison (\ref{ECT-10}) would
reduce to
\begin{eqnarray} \label{ECT-11}
\lambda_{1,p}(B(q,r_0))\geq\lambda_{1,p}\left(\mathcal{B}_{n}(K,r_0)\right),
\end{eqnarray}
 and the equality holds if and only if $B(q,r_0)$ is isometric to
$\mathcal{B}_{n}(K,r_0)$. This eigenvalue comparison (\ref{ECT-11})
and related rigidity result have been shown in \cite[Theorem
1.4]{SPHC} very recently. }
\end{remark}

\vspace{3mm}

\section*{Acknowledgments}
\renewcommand{\thesection}{\arabic{section}}
\renewcommand{\theequation}{\thesection.\arabic{equation}}
\setcounter{equation}{0} \setcounter{maintheorem}{0}

This research was supported in part by the NSF of China (Grant Nos.
11801496 and 11926352), the Fok Ying-Tung Education Foundation
(China), the Key Project of Jiangxi Provincial Natural Science
Foundation (Grant No. 20252BAC250003), Hubei Key Laboratory of
Applied Mathematics (Hubei University), and Key Laboratory of
Intelligent Sensing System and Security (Hubei University), Ministry
of Education.

\end{document}